\documentclass[12pt]{article}
\usepackage[margin=1in]{geometry}

\usepackage{graphicx} 
\usepackage{amsmath, amssymb, amsthm}
\usepackage{MnSymbol}
\usepackage{tikz}
\usepackage{xcolor}

\usetikzlibrary{arrows.meta,calc}

\usepackage{hyperref}
\usepackage{xspace}

\newtheorem{theorem}{Theorem}[section]
\newtheorem{lemma}[theorem]{Lemma}

\newtheorem{conjecture}[theorem]{Conjecture}
\newtheorem{corollary}[theorem]{Corollary}
\newtheorem{proposition}[theorem]{Proposition}
\newcounter{claims}[theorem]
\newtheorem{claim}[claims]{Claim}

\theoremstyle{definition}
\newtheorem{definition}[theorem]{Definition}

\newcommand{\mc}[1]{\mathcal{#1}}

\newcommand{\res}{\!\!\upharpoonright}

\newcommand{\bfSigma}{\boldsymbol{\Sigma}}

\newcommand{\bfPi}{\mathbf{\Pi}}
\newcommand{\compact}{\textup{\textsf{CPCT}}\xspace}
\newcommand{\conn}{\textup{\textsf{CONN}}\xspace}
\newcommand{\localconn}{\textup{\textsf{LC}}\xspace}
\newcommand{\ordered}{\textup{\textsf{ORD}}\xspace}
\newcommand{\btw}{\textup{\textsf{BTW}}}

\newcommand{\twodisconn}{\textup{\textsf{CIRC}}\xspace}
\newcommand{\nondegen}{\textup{\textsf{NDEGEN}}\xspace}
\newcommand{\topspace}{\textsf{TOP}\xspace}

\DeclareMathOperator{\homeo}{homeo}
\DeclareMathOperator{\Mod}{Mod}

\DeclareMathOperator{\HCopies}{HCopies}
\DeclareMathOperator{\Copies}{Copies}

\theoremstyle{remark}
\newtheorem{remark}[theorem]{Remark}

\title{Measuring the complexity of characterizing $[0, 1]$, $S^1$, and $\mathbb{R}$ up to homeomorphism}
\author{Matthew Harrison-Trainor and Eissa Haydar\thanks{Harrison-Trainor was partially supported by the National Science Foundation under Grant DMS-2419591/DMS-2153823. This work began while Haydar was an REU student at the University of Michigan funded under this grant.}}

\begin{document}

\maketitle

\begin{abstract}
    In analogy to the study of Scott rank/complexity of countable structures, we initiate the study of the Wadge degrees of the set of homeomorphic copies of topological spaces. One can view our results as saying that the classical characterizations of $[0,1]$ (e.g., as the unique continuum with exactly two non-cut points, and other similar characterizations), appropriately expressed, are the simplest possible characterizations of $[0,1]$. Formally, we show that the set of homeomorphic copies of $[0,1]$ is $\bfPi^0_4$-Wadge-complete. We also show that the set of homeomorphic copies of $S^1$ is $\bfPi^0_4$-Wadge-complete. On the other hand, we show that the set of homeomorphic copies of $\mathbb{R}$ is $\bfPi^1_1$-Wadge-complete. It is the local compactness that cannot be expressed in a Borel way; the set of homeomorphic copies of $\mathbb{R}$ is Borel within the locally compact spaces.
\end{abstract}

\subsubsection*{AI Disclosure}

The first draft and initial submission of this paper was written without the use of any AI tools. During the refereeing process, ChatGPT Sol 5.6 was used to search for linguistic and mathematical errors which were then fixed.

\section{Introduction}

A long-standing program in computable structure theory is the study of Scott rank and Scott complexity. The basic question is, given a countable structure $\mc{A}$, how hard is it to characterize $\mc{A}$ up to isomorphism among all countable structures? This began with Scott \cite{Scott65} who showed that there is always a formula of the infinitary logic $\mc{L}_{\omega_1 \omega}$ characterizing $\mc{A}$ among countable structures. One can then look for the simplest such sentence, with its complexity called the \textit{Scott rank} of $\mc{A}$.\footnote{There are numerous other definitions of Scott rank, most using some notion of back-and-forth relations, but for the purposes of this paper we can ignore them. See, e.g., \cite{Mon}.}

The particular perspective we will take in this paper is that of \textit{Scott complexity} as introduced in \cite{AGHTT} by Alvir, Greenberg, Harrison-Trainor, and Turetsky. The idea is to consider, in the Polish space $\Mod(\mc{L})$ of $\mc{L}$-structures, the set
\[ \Copies(\mc{A}) = \{ \mc{B} : \mc{B} \cong \mc{A}\} \]
of all isomorphic presentations of $\mc{A}$. By a well-known theorem of Lopez-Escobar \cite{LopezEscobar}, $\mc{A}$ has a $\Sigma_\alpha$ Scott sentence if and only if $\Copies(\mc{A})$ is $\bfSigma^0_\alpha$, and similarly for other complexity classes such as $\Pi_\alpha$/$\bfPi^0_\alpha$ and the difference hierarchy. Thus, instead of considering sentences characterizing $\mc{A}$, we can instead study the Wadge degree of $\Copies(\mc{A})$. This Wadge degree is called the \textit{Scott complexity} of $\mc{A}$, and measures the complexity of characterizing $\mc{A}$ up to isomorphism.

In this paper, our goal is to, instead of characterizing countable structures up to isomorphism, characterize certain topological spaces up to homeomorphism.\footnote{We also note that there is an analogous program for characterizing metric spaces (or indeed, metric structures) up to isometry, by considering $\{ Y : Y \cong_{isometric} X\}$. While we are not aware of any work on this problem specifically, there is existing work along these lines, e.g., \cite{MN13,Lp,Lp2}. The isometry classes are known to be Borel by \cite {MetricScott} or \cite{EFPRTT}.} We restrict our attention to Polish spaces, that is, separable completely metrizable spaces. In this case, there is no syntactic perspective, and so given a topological space $X$ we study the Wadge degree of the homeomorphic copies of $X$,
\[ \HCopies(X) = \{ Y : Y \cong_{\homeo} X\}.\]
We call this the \textit{(topological) Scott complexity} of $X$.  The Scott complexity of $X$ measures the difficulty of characterizing $X$ up to homeomorphism.

There are some options for how we choose to represent our topological spaces, that is, which Polish space we take $\HCopies(X)$ to live in. In descriptive set theory it is common to think of compact spaces as living as closed subsets of the Hilbert cube $[0,1]^{\mathbb{N}}$, which is universal in the sense that any compact Polish space is a closed subspace of the Hilbert cube. We can equip the closed subspaces of the Hilbert cube with the Vietoris topology (metrized by the Hausdorff metric), making it a Polish space. This was essentially the approach taken by Camerlo, Darji, and Marcone in \cite{Camerlo} where they considered the question of the Wadge degree of $\HCopies(X)$ for certain continua $X$ within the Polish space of subcontinua of the Hilbert cube.

Here we take a different approach which has been more popular recently in computability theory, e.g., in the recent book \cite{downeymelnikov}. We think of our topological spaces as being represented as the completion of a countable metric space, just as the real numbers are the completion of $\mathbb{Q}$. The space of all such presentations is itself a Polish space, and we measure the complexity of $\HCopies(X)$ within that space. See Section \ref{sec:two} for details. Morally speaking, the difference is that we have less direct access to information about our space. In the Vietoris topology, the spaces are \textit{effectively compact}, i.e., we have access to, for each $\epsilon$, a finite covering of the space by $\epsilon$-balls. In the setting in this paper we do not have direct access to such covers, which makes the arguments more intricate. After stating our results, we discuss the effect this has.

For a countable structure $\mc{A}$, $\Copies(\mc{A})$ is always Borel. This is not the case for $\HCopies(X)$ if $X$ is not compact; indeed, $\HCopies(\mathbb{R})$ is not Borel, as shown later in this paper. However, if $X$ is compact, $\HCopies(X)$ is Borel; indeed, the homeomorphism equivalence relation is bi-reducible to a complete orbit equivalence relation of a Borel action of a Polish group \cite{ZIELINSKI2016635}. Another perspective, which we do not use in this paper, is to use the effective Gelfand duality \cite{Gelfand} to reduce the problem to isometric isomorphism of unital abelian $C^*$-algebras. These are metric structures and thus subject to the metric Scott analysis \cite {MetricScott}, but it is not yet clear how to fruitfully pull back the metric Scott analysis across the Gelfand duality.

In the case of discrete structures, there are many examples for which the Scott complexity has been computed. For example, the complexity of free groups has been computed \cite{CHKLMMMQW,McCoyWallbaum}: $\mathbb{F}_n$ has complexity \textbf{d-}$\bfSigma^0_2$ for $n$ finite and $\mathbb{F}_\infty$ has complexity $\bfPi^0_4$. Moreover, we know \cite{AGHTT} that the only possible Scott complexities are $\bfSigma^0_\alpha$, $\bfPi^0_\alpha$, and \textbf{d-}$\bfSigma^0_\alpha$.

In the context of topological spaces, we do not yet know any general results about topological Scott complexity. In this paper, we study some of the first 1-dimensional topological spaces to come to mind: the compact line $[0,1]$, the circle $S^1$, and the real line $\mathbb{R}$. We exactly characterize their topological Scott complexity. For these spaces, there are classical characterizations coming from continuum theory, e.g., $[0,1]$ is the unique continuum with exactly two non-cut points. Our setting allows us to analyse such characterizations and determine whether they are as simple as possible. We show that, appropriately expressed,\footnote{For example, to express the notion of a cut point it is helpful to first be in a locally connected space.} they are the simplest possible characterizations.

\begin{theorem}\label{thm:main-arc}
    The unit interval $[0,1]$ has topological Scott complexity $\bfPi^0_4$, i.e., the set
    \[ \HCopies([0,1]) = \{ Y : Y \cong_{\homeo} [0,1] \}\] is $\bfPi^0_4$-Wadge-complete.
\end{theorem}

\begin{theorem}\label{thm:main-circle}
    The circle $S^1$ has topological Scott complexity $\bfPi^0_4$, i.e., the set
    \[ \HCopies(S^1) = \{ Y : Y \cong_{\homeo} S^1 \}\] is $\bfPi^0_4$-Wadge-complete.
\end{theorem}

Camerlo, Darji, and Marcone \cite{Camerlo} also studied these spaces (and more), but as homeomorphism equivalence classes within the space of subcontinua of the Hilbert cube (with the subspace topology induced by the Vietoris topology on the Hilbert cube). They showed that the continua homeomorphic to $[0,1]$ (and to $S^1$) are $\bfPi^0_3$-Wadge-complete in this case. The difference between our results and theirs is that we are comparing our spaces to a larger class of spaces, and our presentations are weaker allowing for non-compact spaces. Given one of our presentations as the completion of a countable metric space, with one jump we can find finite $\epsilon$-covers of the space. We conjecture that one can obtain a jump-inversion result for non-trivial spaces, formally connecting our results and theirs.

\begin{conjecture}
    Let $X$ be a compact Polish space with uncountably-many points. Then:
    \begin{enumerate}
        \item If there is an $A'$-computable presentation of $X$ as a separable metric space which is $A'$-effectively compact (i.e., has $A'$-computable finite $\epsilon$-covers), then there is an $A$-computable presentation of $X$ as a separable metric space.
        \item If
    \[ \{ Y \text{ a closed subspace of the Hilbert cube}  : Y \cong_{\homeo} X \}\]
    is $\bfPi^0_\alpha$-Wadge-complete, then
    \[ \{ Y \text{ a completion of countable metric space }  : Y \cong_{\homeo} X \}\]
     is $\bfPi^0_{\alpha+1}$-Wadge-complete.
    \end{enumerate} 
\end{conjecture}

For the real line $\mathbb{R}$, because it is not compact, the set of homeomorphic copies is not necessarily Borel; in fact, we show that it is $\bfPi^1_1$-Wadge-complete. (We conjecture that this might always be the case for any locally compact non-compact space.) However, we also show that the set of homeomorphic copies of $\mathbb{R}$ is Borel within the locally compact spaces. This suggests that for working with non-compact locally compact spaces one should work within the locally compact spaces.

\begin{theorem}\label{thm:main-R}
    The set
    \[ \HCopies(\mathbb{R}) = \{ Y : Y \cong_{\homeo} \mathbb{R} \}\] is $\bfPi^1_1$-Wadge-complete. It is Borel within the locally compact spaces, and $\bfPi^0_4$-Wadge-complete within the set of proper metrizations of locally compact spaces.
\end{theorem}

The topological spaces $[0,1]$, $S^1$, and $\mathbb{R}$ considered in this paper are the most basic 1-dimensional topological spaces. A natural next class of examples to consider are two-dimensional topological spaces such as $S^2$ and $\mathbb{R}^2$. In \cite{HTMelnikov}, Melnikov and the first author studied closed surfaces, and one of their results is that for $X$ a closed surface, $\HCopies(X)$ is always arithmetic. Thus $\HCopies(S^2)$ is arithmetic, but the upper bound obtained is around $\Sigma^0_{27}$ which is almost certainly not the best possible. There are no known nontrivial lower bounds.\footnote{Since this paper was written, Gamard, Hoyrup, and Terrassin have computed the complexity of $S^2$ \cite{GamardHoyrupTerrassin}, measured in the same way as in \cite{Camerlo}.}

All of the results in this paper are effective; e.g., the $\bfPi^0_4$ sets are actually lightface $\Pi^0_4$, and the Wadge reductions are not just continuous but computable. In particular we could also state theorems for index sets, e.g., the index set \[\{ i \mid \text{the $i$th computable Polish space is homeomorphic to $[0,1]$}\}\] is $\Pi^0_4$ $m$-complete.

\section{Background}

\subsection{Presentations of separable metric spaces}\label{sec:two}

The traditional approach to studying the class of $\mathcal{L}$-structures in computable structure theory is to consider the space of \textit{representations} of $\mathcal{L}$-structures with domain $\mathbb{N}$. Such a representation can be thought of as an encoding of the structure in machine-readable form, and the space of representations forms a Polish space $\text{Mod}(\mathcal{L})$.

In this work, we want to consider representations of topological (and more specifically Polish) spaces. Since Polish spaces typically have uncountably many points, we cannot represent each point by a natural number. Turing~\cite{Turing:36,Turing:37} had already realized this long ago, and instead suggested representing a real number by a Cauchy sequence of rational numbers. Moreover, in order to perform any computations, we need to know how fast the Cauchy sequence converges.

\begin{definition}
    A Cauchy sequence $(x_i)_{i \in \mathbb{N}}$ in a metric space $(M,d)$ is \textit{fast} if $d(x_i,x_{i+1}) < 2^{-i}$.
\end{definition}

One thinks of the reals $\mathbb{R}$ presented as a completion of the rationals $\mathbb{Q}$, with each real number $r$ represented by the Cauchy sequences converging to $r$. We extend this to arbitrary Polish spaces, using any countable dense subset of the space in place of $\mathbb{Q}$.

\begin{definition}
    A \textit{presentation} of a Polish space $X$ is a countable set $M$ of points (identified with $\mathbb{N}$) and a metric $d$ such that $X$ is homeomorphic to the completion $\overline{M}$ of the metric space $(M,d)$. We think of representing the metric $d$ by giving, for every two elements $x,y$ of $M$, a fast Cauchy sequence converging to $d(x,y)$. Thus, a presentation of $X$ can be identified with an infinite binary string.
\end{definition}

Given a presentation $X = \overline{M}$ of a Polish space, we refer to the points of $M$ as \textit{special points}. We represent the non-special points of $X$ using fast Cauchy sequences from $M$. The presentation $\mathbb{R} = \overline{\mathbb{Q}}$ is the natural presentation of $\mathbb{R}$ using $\mathbb{Q}$ as the special points, but there are, of course, other possible presentations of $\mathbb{R}$.


When we talk about the topology on $X$, we will generally use the open balls centered on special points with rational radii. We call these the \textit{basic open balls} and denote them by $B_r(x)$. We note that the basic open balls might not even be connected, even when the space is, e.g., connected and locally path connected. We denote the basic closed balls by $\overline{B}_r(x)$, noting that this is different from the closure of the basic open ball.

We will frequently make use of \textit{$\epsilon$-paths}. An $\epsilon$-path from $x$ to $y$ is a sequence of points $x = u_1,\ldots,u_n = y$ such that $d(u_i,u_{i+1}) < \epsilon$.  The beginning and ending points $x$ and $y$ need not be special points, but for convenience we require that all the points in between them must be special. The $\epsilon$-path is within a set $U$ if all of the points are within $U$. We say that it goes through $U$ if all of its points other than its endpoints lie in $U$. If $U$ is open, then there is an $\epsilon$-path (using special points) from $x$ to $y$ within $U$ if and only if there is an $\epsilon$-path (not necessarily consisting of special points) from $x$ to $y$ within $U$.

The space of all presentations of Polish spaces is a $\Pi^0_2$ subset of $2^{ \omega}$ (under the identification given above) and so is itself a Polish space in its own right. We call this space $\topspace$. Note that $\topspace$ is homeomorphic to Baire space $\omega^\omega$, and so as described in the following subsection we can thus talk about Wadge reducibility and the Wadge degrees of subsets of $\topspace$.

\subsection{Wadge degrees and reductions}

\begin{definition}[Wadge]
	Let $A$ and $B$ be subsets of Baire space $\omega^\omega \cong \topspace$. We say that $A$ is \emph{Wadge reducible} to $B$, and write $A \leq_W B$, if there is a continuous function $f$ with $A = f^{-1}[B]$, i.e.
	\[x \in A \Longleftrightarrow f(x) \in B.\]
\end{definition}

\noindent The equivalence classes under this pre-order are called the Wadge degrees. The Wadge hierarchy is the set of Wadge degrees under continuous reductions. With enough determinacy, the Wadge hierarchy is very well-behaved; it is well-founded (as proved by Martin and Monk) and almost totally ordered (in the sense that any anti-chain has size at most two) \cite{Wadge}.

Pointclasses such as $\bfSigma^0_\alpha$, $\bfPi^0_\alpha$, $\bfPi^1_\alpha$, and $\bfSigma^1_\alpha$ are all closed downwards in the Wadge degrees. For $\Gamma$ a pointclass, and $A$ a set, we say that $A$ is $\Gamma$-Wadge-complete (or simply $\Gamma$-complete) if it is $\Gamma$ and also $\Gamma$-hard, which means that for every $B \in \Gamma$, $B \leq_W A$.

The reader may be a computability theorist used to many-one reductions of index sets. All of our Wadge reductions will actually be computable reductions, and so the same completeness theorems hold for many-one reductions. But building a Wadge reduction is more general (e.g., there are computable structures such that the set of indices of computable copies is $\Sigma^0_3$, but the set of indices of $0'$-computable copies is d-$\Sigma^0_2$ relative to $0'$). In this paper every upper bound on the complexity of a set will be a lightface upper bound.

Sometimes the complexity of a particular set is dominated by one aspect. In computable structure theory, one such example is finitely generated groups. The class of all finitely generated groups is $\Sigma^0_3$, and there are particular finitely generated groups $G$ with $\Copies(G)$ being $\bfSigma^0_3$-complete \cite{HTHo}. On the other hand, for any finitely generated group $G$, $\Copies(G)$ is the intersection of a $\Pi^0_3$ set and the set of all finitely generated groups. Thus the complexity of saying that $G$ is finitely generated ``drowns out'' the complexity of saying which finitely generated group it is. Another example from \cite{CHKLMMMQW,McCoyWallbaum} is that, with $F_\infty$ the free group on countably infinitely many generators, $\Copies(F_\infty)$ is $\bfPi^0_4$-complete, but $\bfPi^0_3$-complete within the class of free groups. The standard way to separate these is to measure the complexity of $\Copies(G)$ within the set of finitely generated groups/free groups, as follows.

\begin{definition}
	Let $A$ and $C$ be sets in $\omega^\omega$.
 \begin{enumerate}
     \item We say that $A$ is $\Gamma$ within $C$ if there is a set $A^* \in \Gamma$ with $A = A^* \cap C$.
      \item We say that $A$ is $\Gamma$-Wadge-complete (or simply $\Gamma$-complete) within $C$ if $A$ is $\Gamma$ within $C$ and for every $B \in \Gamma$, there is a continuous function $f$ such that for all $x$, $f(x) \in C$, and $x \in B \Longleftrightarrow f(x) \in A$.
 \end{enumerate} 
\end{definition}

The main example of this phenomenon in our paper is in Theorem \ref{thm:main-R} where when characterizing $\mathbb{R}$ up to homeomorphism, the fact that it is locally compact (which is $\Pi^1_1$) drowns out everything else. We show that $\HCopies(\mathbb{R})$ is Borel within the class of locally compact spaces.

The second example is that $\HCopies([0,1])$ is $\bfPi^0_4$-complete, but the most difficult part is in expressing local connectedness. In Theorem \ref{lem:01-within} we show that $\HCopies([0,1])$ is actually $\bfPi^0_3$-complete within the compact, connected, locally connected spaces.

\section{\texorpdfstring{A $\Pi^0_4$ characterization of $[0,1]$}{A Pi04 characterization of [0,1]}}\label{sec:three}

In this section we prove the first part of Theorem \ref{thm:main-arc}, that $\HCopies([0,1])$ is $\Pi^0_4$. We prove the second half of the theorem, that $\HCopies([0,1])$ is $\bfPi^0_4$-hard, in the next section.

We recall the classical characterization of $[0,1]$ as the unique continuum (non-empty compact connected metric space) with exactly two non-cut points \cite[Theorem 6.17]{Nadler92}. It is not obvious how to express the existence of non-cut points as we are limited to quantifying over special points. Instead, we will express an idea of ``betweenness.'' In the following subsections, we will express various properties of a topological space in a $\Pi^0_4$ (or often simpler) way, and then show that they characterize $[0,1]$.

\subsection{Compactness}

Following \cite{MN13}, compactness can be expressed in a $\Pi^0_3$ way in terms of total boundedness: We say that a countable metric space satisfies \compact if and only if for every $\epsilon > 0$, there is $n \in \mathbb{N}$ and $n$ special points $x_1,\ldots,x_n$ such that for all special points $y$, $y$ is within $\epsilon$ of some $x_i$. We take a second here to note that $d(y,x_i) < \epsilon$ is a $\Sigma^0_1$ condition, because in a presentation we are only given a fast Cauchy sequence converging to $d(y,x_i)$.  To make \compact $\Pi^0_3$ we must instead interpret ``$y$ is within $\epsilon$ of some $x_i$'' as that ``$d(y,x_i) \leq \epsilon$'' which is a $\Pi^0_1$ condition. As a general principle, these types of conditions are really about \textit{sufficient smallness} and so asking for strict or non-strict inequalities are equivalent.

\subsection{Connectedness}

Connectedness for arbitrary spaces is not Borel, even in the case of finite dimensional Euclidean spaces, as shown in \cite{DebsSR}. However, connectedness in a compact space can be expressed much more simply.

We say that $X = \overline{(M,d)}$ satisfies $\conn$ if for every $\epsilon > 0$ and special points $x_1,\ldots,x_m$ and $y_1,\ldots,y_m$, if $U = \overline{B}_{\epsilon}(x_1) \cup \cdots \cup \overline{B}_{\epsilon}(x_m)$ and $V = \overline{B}_{\epsilon}(y_1) \cup \cdots \cup \overline{B}_{\epsilon}(y_m)$ cover the special points (i.e., for all special points $z$, there is some $x_i$ with $d(z,x_i) \leq \epsilon$ or some $y_i$ with $d(z,y_i) \leq \epsilon$) the sets $U^* = B_{2\epsilon}(x_1) \cup \cdots \cup B_{2\epsilon}(x_m)$ and $V^* = B_{2\epsilon}(y_1) \cup \cdots \cup B_{2\epsilon}(y_m)$ are not disjoint (i.e., there is a special point $z$ such that for some $x_i$ and $y_j$, $d(z,x_i) < 2\epsilon$ and $d(z,y_j) < 2\epsilon$). Note that $\conn$ is a $\Pi^0_2$ property.\footnote{It is easy to see that the set of connected spaces is also $\Pi^0_2$-complete within the compact spaces; one builds a copy of the space $[0,1-\frac{1}{n}] \cup [1,2]$, and lets $n \to \infty$ in the infinitary outcome.} We remark again that to make this $\Pi^0_2$ we were careful to use the closed condition in the antecedent of the implication and the open condition in the consequence. Going forward, we will no longer explicitly comment on such issues.

\begin{lemma}
    Let $X = \overline{(M,d)}$ be compact. Then $X$ is connected if and only if $(M,d)$ satisfies $\conn$.
\end{lemma}
\begin{proof}
    We argue that a compact space $X$ is connected if and only if it satisfies $\conn$. Suppose that $X$ is not connected, with $X = U \sqcup V$ being a partition. Let $d$ be the distance between $U$ and $V$, and choose $\epsilon > 0$ such that $4 \epsilon < d$. Let $U = \overline{B}_{\epsilon}(x_1) \cup \cdots \cup \overline{B}_{\epsilon}(x_m)$ and $V = \overline{B}_{\epsilon}(y_1) \cup \cdots \cup \overline{B}_{\epsilon}(y_m)$ be $\epsilon$-covers of $U$ and $V$ by closed balls centered at special points. Note that $U^* = B_{2\epsilon}(x_1) \cup \cdots \cup B_{2\epsilon}(x_m)$ and $V^* = B_{2\epsilon}(y_1) \cup \cdots \cup B_{2\epsilon}(y_m)$ are still disjoint by choice of $\epsilon$. Thus we witness the negation of $\conn$.

    On the other hand, suppose that $X$ fails to satisfy $\conn$. Take a witness with $U^* = B_{2\epsilon}(x_1) \cup \cdots \cup B_{2\epsilon}(x_m)$ and $V^* = B_{2\epsilon}(y_1) \cup \cdots \cup B_{2\epsilon}(y_m)$. Then $U^*$ and $V^*$ are disjoint, since if they intersected then there would be a special point $z$ in their intersection, as they are open. They also cover $X$, because $U = \overline{B}_{\epsilon}(x_1) \cup \cdots \cup \overline{B}_{\epsilon}(x_m)$ and $V = \overline{B}_{\epsilon}(y_1) \cup \cdots \cup \overline{B}_{\epsilon}(y_m)$ cover the special points of $X$, and every point of $X$ is within $\epsilon$ of a special point of $X$. Thus $X = U^* \sqcup V^*$ witnesses that $X$ is not connected.
\end{proof}

\subsection{Local connectedness}

We expect that, in general, local connectedness might be quite hard to express (like local compactness, which is not Borel \cite{NiesSolecki}). However, we do not have to find a condition which is both necessary and sufficient for local connectedness, but just a condition equivalent to local connectedness within compact spaces. Our condition is as follows.
\begin{description}
    \item[\textnormal{\localconn:}] For all basic open sets $A = B_{r}(c)$ and $B = B_{2r}(c)$ there are special points $x_1,\ldots,x_n \in A$ and $\epsilon > 0$ such that:
    \begin{enumerate}
        \item for all special points $y \in A$, there is some $i$ such that there is an $\epsilon$-path from $y$ to $x_i$ in $B$.
        \item for distinct $i,j$ there is no $\epsilon$-path from $x_i$ to $x_j$ in $B$.
        \item for all $i$ and for all special points $y \in A$ and $\epsilon' > 0$ if there is an $\epsilon$-path from $y$ to $x_i$ in $B$ then there is an $\epsilon'$-path from $y$ to $x_i$ in $B$.
    \end{enumerate}
\end{description}

This is $\Pi^0_4$. First, we show that our condition \localconn implies local connectedness via weak local connectedness. We recall the difference between locally connected and weakly locally connected spaces. Let $X$ be a topological space. $X$ is locally connected if every point has a neighbourhood base consisting of open connected sets. $X$ is weakly locally connected
if every point has a neighbourhood base consisting of connected sets, which are not necessarily open. However these are equivalent; see, e.g., Theorem 27.16 of \cite{MR0264581}.

\begin{lemma}
    If $X = \overline{(M,d)}$ is compact and satisfies \localconn, then $X$ is weakly locally connected and hence locally connected.
\end{lemma}
\begin{proof}
    Let $x$ be any point, and $U$ an open set containing $x$. We can choose basic open balls $x \in B_r(c) \subseteq \overline{B_r(c)} \subseteq B_{2r}(c) \subseteq \overline{B_{2r}(c)} \subseteq U$. Let $A = B_r(c)$ and $B = B_{2r}(c)$. We will find a connected neighbourhood $D$ of $x$ contained in $B$. To witness that $D$ is a neighbourhood, we will have $D \cap A$ open.
    
    By \localconn, there are $x_1,\ldots,x_n \in A$ and $\epsilon$ satisfying (1), (2), and (3). Define
    \[ D_i = \{ y \in \overline{B} : \text{for all $\epsilon' > 0$ there is an $\epsilon'$-path from $y$ to $x_i$ through $B$} \}.\]
    Clearly $D_i \subseteq \overline{B} \subseteq U$. We prove the following five claims.

    \begin{claim}
        Each $D_i$ is closed.    
    \end{claim}
    \begin{proof}
        Given $y \in \overline{B}$, suppose that $y$ is a limit point of $D_i$. We argue that $y \in D_i$. Given $\epsilon' > 0$, there is some point $y'$ of $D_i$ with $d(y,y') < \epsilon'/2$. Then there is an $\epsilon'/2$-path from $y'$ to $x_i$ through $B$, which, by replacing $y'$ by $y$, is an $\epsilon'$-path from $y$ to $x_i$ through $B$.
    \end{proof}

    \begin{claim}
        Each $D_i$ is connected.    
    \end{claim}
    \begin{proof}
        Let $K = \overline{B}$. For $\epsilon > 0$, let $E_\epsilon$ be the set of points which can be joined to $x_i$ by an $\epsilon$-path in $K$. We argue that
        \[ D_i = \bigcap_{\epsilon > 0} E_\epsilon. \]
        Clearly $D_i$ is contained in the intersection. Towards the other containment, suppose that $y \in \bigcap_{\epsilon > 0} E_\epsilon$, and fix $\epsilon > 0$. Choose an $\epsilon/3$-path $y = y_0,\ldots,y_m = x_i$ in $K$. Since the special points in $B$ are dense in $K$, we may replace each intermediate $y_j$ by a special point of $B$ within $\epsilon/3$ of $y_j$. The resulting sequence is an $\epsilon$-path from $y$ to $x_i$ through $B$, so $y \in D_i$.

        Each $E_\epsilon$ is clopen in $K$. It is open because a point within $\epsilon$ of a point in $E_\epsilon$ can be appended to an $\epsilon$-path, and its complement is open by the same argument. Thus $\bigcap_{\epsilon > 0} E_\epsilon$ is the quasi-component of $x_i$ in $K$. Indeed, every $E_\epsilon$ is a clopen neighbourhood of $x_i$; conversely, if $F$ is any clopen neighbourhood of $x_i$ in $K$, compactness gives a positive distance between $F$ and $K \setminus F$, so $E_\epsilon \subseteq F$ for all sufficiently small $\epsilon$. Since quasi-components and connected components agree in compact Hausdorff spaces, $D_i$ is the connected component of $x_i$ in $K$, and is therefore connected.
    \end{proof}

    \begin{claim}
        The $D_i$ are pairwise disjoint.
    \end{claim}
    \begin{proof}
        Suppose that there is some $y \in D_i \cap D_j$. Then there is an $\epsilon/2$-path from $y$ to $x_i$ through $B$ and an $\epsilon/2$-path from $y$ to $x_j$ through $B$. Putting them together, and deleting $y$, we get an $\epsilon$-path from $x_i$ to $x_j$ in $B$. This contradicts (2).
    \end{proof}

    \begin{claim}
        Each $D_i \cap A$ is open. Indeed, given $y \in D_i \cap A$ and $z \in A$ with $d(y,z) < \epsilon / 2$, $z \in D_i$.
    \end{claim}
    \begin{proof}
        Choose some $\epsilon'$. We wish to find an $\epsilon'$-path from $z$ to $x_i$ in $B$, which would place $z \in D_i$. Since $y \in D_i$, we have an $\epsilon$-path from $y$ to $x_i$ in $B$. Since the special points are dense, we may choose a path among them, other than the $y$, which may be non-special. Call the path-elements $a_i$, where $a_1 = y$ and $a_n = x_i$. We have that $y$ is in $B_{\epsilon / 2}(z)$ and $B_\epsilon(a_2)$, and since both of these are open, there is some ball around $y$ in both of these as well. Choose in particular a special $y' \in B$ within these two and also within $\epsilon'$ of $y$. Choose also a special $z' \in A$ with $d(y,z') < \epsilon / 2$ and $d(z,z') < \epsilon'$. Then we have an $\epsilon$-path $z',y',a_2,\ldots,a_n=x_i$ among the special points, which by (3) gives us an $\epsilon'$-path from $z'$ to $x_i$ in $B$, which by our choice of $z'$ gives us an $\epsilon'$-path from $z$ to $x_i$ in $B$. Thus $z \in D_i$. Together with the fact that $A$ is open, this shows that $D_i \cap A$ is open.
    \end{proof}

    \begin{claim}
        The $D_i$ cover $A$.
    \end{claim}
    \begin{proof}
        For any $y \in A$, choose some special $y' \in A$ with $d(y,y') < \epsilon / 2$. By the ``Indeed'' statement in the above claim, we have that $y' \in D_i \Rightarrow y \in D_i$ for any $i$. Since $y'$ is special, (1) and (3) give us that $y' \in D_i$ (whence $y \in D_i$) for some $i$.
    \end{proof}

    Now since the $D_i$ cover $A$, $x \in D_i$ for some $i$. Then $D_i$ is a connected neighbourhood of $x$ contained in $U$. This completes the proof of the lemma.
\end{proof}

Second, we will show that any compact and locally connected space satisfies the condition \localconn.

\begin{lemma}

    If $X = \overline{(M,d)}$ is compact and locally connected then it satisfies \localconn.
\end{lemma}
\begin{proof}
    Firstly, take basic open sets with $A = B_{r}(c)$ and $B = B_{2r}(c)$ for some $c$ and $r$. We want to show that there is a finite set of $x_i$ together with an $\epsilon$ that satisfy (1)-(3) of LC.

    Let $C_1,C_2,\ldots$ be the connected components of $B$, which are open sets because $B$ is locally connected.
    Define an equivalence relation $\sim$ on the $C_i$ saying that $C_i \sim C_j$ if for every $\epsilon' > 0$ and $x \in C_i$ and $y \in C_j$, there is an $\epsilon'$-path from $x$ to $y$ through $B$. (Note that this is true of any two points $x,y$ in the same $C_i$.) For each equivalence class, define an open set $D$ which is the union of the $C_i$ in that equivalence class. Let $D_1,D_2,\ldots$ be these sets.

    Each $D_i$ is open as the union of (open) connected components of $B$. The $D_i$ cover $B$, and since $\overline{A} \subseteq B$ is compact, they cover $\overline{A}$. There are finitely many of them, say $D_1,\ldots,D_n$, which cover $\overline{A}$. We may discard any sets that are disjoint from $A$.

    For distinct $i,j$ there is $\epsilon_{i,j}>0$ such that no $\epsilon_{i,j}$-path in $B$ joins a point of $D_i$ to a point of $D_j$. Otherwise, for each $\epsilon > 0$ there would be an $\epsilon$-path through $B$ joining a point of $D_i$ to a point of $D_j$. Equivalently, there would be a point of $D_i$ and a point of $D_j$ such that for each $\epsilon > 0$ there is an $\epsilon$-path through $B$ joining those two points, which of course is not the case. Choose $\epsilon = \min_{i \neq j} \epsilon_{i,j}$ (or $\epsilon = 1$ if $n = 1$). Choose special points $x_i \in A \cap D_i$ for each $i$.

    \begin{claim}
        $\epsilon$ and $x_1,\ldots,x_n$ satisfy (1).
    \end{claim}
    \begin{proof}
        Since the $D_i$ cover $\overline{A}$, they cover $A$, and so since $y \in A$, $y \in D_i$ for some $i$. By the construction of the $D_i$ and since $x_i \in D_i$, this means that there is an $\epsilon'$-path from $y$ to $x_i$ through $B$ for all $\epsilon' > 0$, including $\epsilon$.
    \end{proof}

    \begin{claim}
        $\epsilon$ and $x_1,\ldots,x_n$ satisfy (2).
    \end{claim}
    \begin{proof}
        Suppose that the converse of (2) was true. Then for some $i\neq j$, there would be an $\epsilon$-path through $B$ from $x_i$ to $x_j$. But this contradicts the choice of $\epsilon$ in the construction above.
    \end{proof}

    \begin{claim}
        $\epsilon$ and $x_1,\ldots,x_n$ satisfy (3).
    \end{claim}
    \begin{proof}
        This follows from the construction of the $D_i$ as the union of connected components which have $\epsilon$-paths through $B$ between any two elements for all $\epsilon$. By the choice of $\epsilon$ in the above construction, $y$ having an $\epsilon$-path through $B$ to some $x_i$ places $y$ in $D_i$, which then, for any $\epsilon'$, has an $\epsilon'$-path through $B$ to any other point within $D_i$, including $x_i$.
    \end{proof}

    The claims prove the lemma.
\end{proof}

If $X$ is locally connected, compact, and connected, then it is path-connected and locally path-connected \cite[6.3.11]{MR1039321}.

\begin{corollary}
    If $X = \overline{(M,d)}$ satisfies \compact, \conn, and \localconn then it is path-connected and locally path-connected.
\end{corollary}

\subsection{Betweenness}

In a standard proof of the classification of $[0,1]$ as the unique continuum with exactly two non-cut points, the separation ordering plays an important role. We will deal not with the ordering, but with the corresponding notion of betweenness where one point is between two others if removing it separates the two points.

\begin{definition}
    Define a ternary betweenness relation $\btw(x,y,z)$ by setting $\btw(x,y,z)$ if $x,y,z$ are all distinct and for all $\delta > 0$, there are $\delta',\epsilon < \delta$ such that for all $\epsilon$-paths $x = a_1,\ldots,a_n = z$ from $x$ to $z$, there is some $i$ such that $d(y,a_i) \leq \delta'$.
\end{definition}

If $\btw(x,y,z)$ holds, then we say that \textit{$y$ is between $x$ and $z$}. It is easy to see that there is a symmetry between the first and third arguments, i.e., $\btw(x,y,z)$ if and only if $\btw(z,y,x)$. Note that the betweenness relation is $\Pi^0_3$.

\begin{definition}
    Let \ordered express that for all distinct special points $x,y,z$, either $\btw(x,y,z)$, $\btw(y,x,z)$, or $\btw(x,z,y)$.
\end{definition}

This expresses totality of the betweenness relation, and is $\Pi^0_3$.

\begin{lemma}
    If $X = \overline{(M,d)}$ is homeomorphic to $[0,1]$ then $X$ satisfies \ordered.
\end{lemma}
\begin{proof}
    Take distinct $x, y, z \in X$. Since $X \cong [0,1]$, we can relabel $x$, $y$, and $z$ so $y$ is between $x$ and $z$. We wish to show that for all rational $\delta$, there exist $\delta', \epsilon < \delta$ such that any $\epsilon$-path from $x$ to $z$ goes within $\delta'$ of $y$. Take any $\delta' < \delta$ such that $x,z \notin B_{\delta'}(y)$. Then in $X - B_{\delta'}(y)$, the points $x$ and $z$ are in disjoint closed sets. Let $\epsilon < \delta$ be smaller than the distance between these two closed sets. Then there is no $\epsilon$-path from $x$ to $z$ in $X - B_{\delta'}(y)$, hence any $\epsilon$-path from $x$ to $z$ goes within $\delta'$ of $y$.
\end{proof}

\subsection{Putting everything together: arcs and cutpoints}

So far, none of the properties we have introduced stipulate that a space must be non-trivial, so we introduce one final property. Recall that a continuum is said to be non-degenerate if it has more than one point. We say that $X = \overline{(M,d)}$ satisfies \nondegen if it has at least two distinct points. This is $\Sigma^0_2$.

Combining various previous lemmas, we have the following.

\begin{lemma}
    If $X = \overline{(M,d)}$ is homeomorphic to $[0,1]$, then it satisfies \nondegen, \compact, \conn, \localconn, and \ordered.
\end{lemma}

We must now show the converse.

\begin{lemma}\label{claim:firstclaim}
    If $X = \overline{(M,d)}$ satisfies \nondegen, \compact, \conn, \localconn, and \ordered, then it is homeomorphic to $[0,1]$.
\end{lemma}
\begin{proof}
Since $X = \overline{(M,d)}$ satisfies \compact, \conn, and \localconn we know that it is compact, connected, locally connected, path connected, and locally path connected. In particular, $X$ is a continuum. We will show that $X$ has exactly two non-cut points, which implies that it is homeomorphic to $[0,1]$. (See \cite[Theorem 6.17]{Nadler92}.)



\begin{claim}\label{lem:arc-avoid}
    Suppose that there is an arc from $x$ to $z$ that does not pass through $y$. Then $\neg\btw(x,y,z)$.
\end{claim}
\begin{proof}
    Let $f$ be the arc from $x$ to $z$ and suppose that $y$ is not on this arc. There is some distance $d$ between $y$ and $f$. For sufficiently small $\epsilon$, an $\epsilon$-path from $x$ to $z$ approximating $f$ avoids $B_{d/2}(y)$. Thus $\neg \btw(x, y, z)$.
\end{proof}

\begin{claim}\label{lem:unique-arcs}
    For $x, y \in X$, there is a unique arc from $x$ to $y$.
\end{claim}
\begin{proof}
    Much of the technical difficulty with proving this claim is the fact that \ordered only tells us about special points, but to prove the claim we need to argue for all points $x,y$ not necessarily special.
    
    The existence of such an arc follows from path-connectedness (which is equivalent to arc-connectedness in Hausdorff spaces). To prove uniqueness, suppose we have two distinct arcs, $f$ and $g$, from $x$ to $y$. Take a point $z$ on $f$ that is not on $g$. Take a path-connected set around $z$ that is disjoint from $g$, say $U$, and take a special point $z'\in U$. Then there is an arc $h$ from $z$ to $z'$ within $U$, and thus disjoint from $g$. We may assume by moving $z$ and shortening $h$ that $z$ is the only point on both $h$ and $f$. Break up $f$ at the point $z$ into two segments $f_1$ from $x$ to $z$ and $f_2$ from $z$ to $y$. Choose also special points $x'$ and $y'$ close to $x$ and $y$, respectively, with arcs from $x$ to $x'$ near $x$ and from $y$ to $y'$ near $y$. The picture is something like this:
    
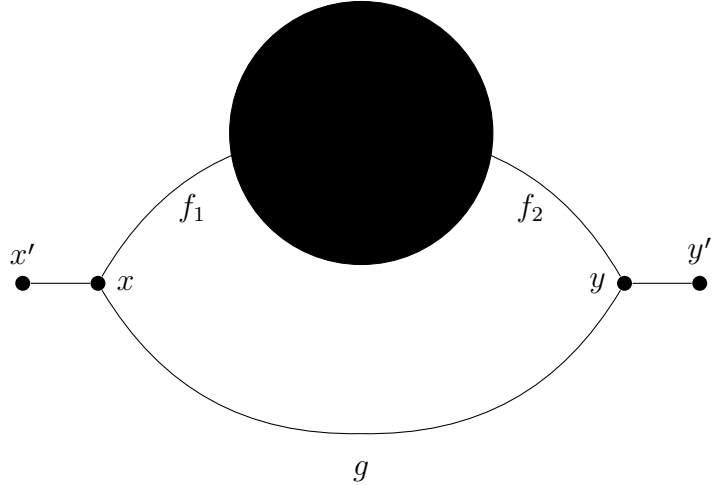
\[ \begin{tikzpicture}

        \node [fill=black,label=right:$x$,circle,inner sep=2pt](x)  at (0,0) {};

        \node [fill=black,label=$x'$,circle,inner sep=2pt](xp)  at (-1,0) {};

        \node [fill=black,label=left:$y$,circle,inner sep=2pt](y)  at (7,0) {};

        \node [fill=black,label=$y'$,circle,inner sep=2pt](yp)  at (8,0) {};

        \node [fill=black,label=below:$z$,circle,inner sep=2pt](z)  at (3.5,2) {};

        \node [fill=black,label=$z'$,circle,inner sep=2pt](zp)  at (3.5,3) {};

        \draw (x) to [bend left] (z);
        \draw (z) to [bend left] (y);

        \draw (x) to [bend right] (3.5,-2);
        \draw (3.5,-2) to [bend right] (y);

        \node [fill=none](f1)  at (5.75,1) {$f_2$};
        \node [fill=none](f2)  at (1.25,1) {$f_1$};
        \node [fill=none](g)  at (3.5,-2.5) {$g$};

        \draw (xp) -- (x);
        \draw (yp) -- (y);
        \draw (zp) -- (z);

        \draw[fill=black, opacity=0.2] (z) circle [radius=1.75];

        \node [fill=none](U)  at (2.5,2.5) {$U$};
        
    \end{tikzpicture} \]
    Thus we have arcs containing $x'$ and $y'$ but not $z'$, $y'$ and $z'$ but not $x'$, and $x'$ and $z'$ but not $y'$. By Claim \ref{lem:arc-avoid} we have that $\neg \btw(y', z', x')$, $\neg \btw(x', y', z')$, and $\neg \btw(z', x', y')$. This contradicts $\ordered$ and proves the claim.
\end{proof}

\begin{claim}
    There are at most two non-cut points. 
\end{claim}
\begin{proof}
    Suppose that there were three non-cut points $a,b,c$. Without loss of generality, we may assume that the arc $f$ from $a$ to $b$ does not contain $c$. (If it did, then take the shorter arc from $a$ to $c$ and swap $b$ and $c$.) We may also assume that the arc $g$ from $b$ to $c$ does not contain $a$, as if it did, then we could take the arc from $a$ to $c$ and swap $a$ and $b$.

    Now we argue that $b$ is in fact a cut point. Suppose not; then $X - \{b\}$ is connected, and since it is still locally path connected, it is path connected. Thus there is an arc from $a$ to $c$ avoiding $b$. Since there is a unique arc between any two points, these points can only possibly be arranged as follows:
\[
\begin{tikzpicture}[scale=0.8]
    \draw[thick] (0,0) -- (-1.5,1.5);
    \draw[thick] (0,0) -- (1.5,1.5);
    \draw[thick] (0,0) -- (0,-1.5);

    \node[fill=black,circle,inner sep=1.5pt,label=left:$a$] at (-1.5,1.5) {};
    \node[fill=black,circle,inner sep=1.5pt,label=right:$b$] at (1.5,1.5) {};
    \node[fill=black,circle,inner sep=1.5pt,label=below:$c$] at (0,-1.5) {};
    \node[fill=black,circle,inner sep=1.5pt] at (0,0) {};
\end{tikzpicture}
\]
Choosing special points on each arm, this does not satisfy \ordered.
\end{proof}

It is a theorem of Moore that any non-degenerate continuum must have at least two non-cut points (see \cite[Theorem 6.6]{Nadler92}). Thus we have exactly two non-cut points. This proves the lemma.
\end{proof}

Putting together all of the lemmas, we have one half of Theorem \ref{thm:main-arc}.
\begin{theorem}
   $X = \overline{(M,d)}$ satisfies \nondegen, \compact, \conn, \localconn, and \ordered if and only if $X$ is homeomorphic to $[0,1]$. In particular, $\HCopies([0,1])$ is $\Pi^0_4$.
\end{theorem}

Note that of all of these conditions, it is only \localconn which is $\Pi^0_4$. Thus, in some sense, the hard part of saying that $X = \overline{(M,d)}$ is homeomorphic to $[0,1]$ is in saying that $X$ is locally connected.

\begin{corollary}\label{cor:within-loc-conn}
    $\HCopies([0,1])$ is $\Pi^0_3$ within the locally connected spaces.
\end{corollary}

In Theorem \ref{lem:01-within} to follow, we will show that $\HCopies([0,1])$ is $\Pi^0_3$-complete within the locally connected spaces.

\section{\texorpdfstring{A warmup: $\bfPi^0_3$-hardness within the compact, connected, locally connected spaces}{A warmup: Pi03-hardness within the compact, connected, locally connected spaces}}

We will prove the following theorem as a warmup.

\begin{theorem}\label{lem:01-within}
    $\HCopies([0,1])$ is $\bfPi^0_3$-Wadge-complete within the compact, connected, locally connected spaces.
\end{theorem}
\begin{proof}
    It follows from Corollary \ref{cor:within-loc-conn} that $\HCopies([0,1])$ is $\Pi^0_3$ within the compact locally connected spaces. So it remains to prove a hardness result. Consider the $\bfPi^0_3$-Wadge-complete set $S = \{W \subseteq \omega \times \omega \; \mid \; \forall n \;\; |W_n| < \infty\}$ where $W_n = \{ m \mid (n,m) \in W\}$ is the $n$th column of $W$. Given $W$ we must build a compact, connected, locally connected space $X = X_W$ such that $W \in S$ if and only if $X_W \cong [0,1]$. Fix $W$ for which we must define $X = X_W$. We will work somewhat dynamically by letting $W_n[i] = W_n \res i = \{m \in W_n \mid m < i\}$. Note that $W_n[0] = \varnothing$.

    We construct $X$ inside the unit square $[0,1] \times [0,1]$. We will always include in $X$ the line $[0,1] \times \{0\} = \{(x,0) : x \in [0,1]\}$. We will also include in $X$ the graph $\Gamma_f$ of a continuous function $f \colon [0,1] \to [0,1]$. Thus 
    \[ X = \Gamma_f \; \cup \; [0,1] \times \{0\}.\]
    The function $f$ will have the property that $f(0) = 0$ and $f(1) = 1$. Thus $X$ will always be connected, locally connected, and compact. $X$ will be homeomorphic to $[0,1]$ if and only if $f(x) > 0$ for all $x > 0$.

    It just remains to define $f$. Pick some sequence of elements, alternating between $a_i$ and $b_i$,
    \[ 0 < \cdots < b_2 < a_2 < b_1 < a_1 < b_0 < a_0 = 1 \]
    converging to $0$. In each interval $[a_{i+1},a_i]$ choose an increasing sequence 
    \[ a_{i+1} = \ell_i^0 < \ell_i^1 < \ell_i^2 < \cdots < b_i \]
    converging to $b_i$, and a decreasing sequence
    \[ b_i < \cdots < r_i^2 < r_i^1 < r_i^0 = a_i \]
    converging to $b_i$. Thus we have divided the interval $[0,1]$ into intervals $[a_{i+1},a_i]$, and we have divided each of these intervals into a left interval $[a_{i+1},b_i]$ and a right interval $[b_i,a_i]$. The left and right intervals are in turn divided into intervals $[\ell_i^j,\ell_i^{j+1}]$ on the left and $[r_i^{j+1},r_i^j]$ on the right:
    \[ a_{i+1} = \ell_i^0 < \ell_i^1 < \ell_i^2 < \cdots < b_i < \cdots < r_i^2 < r_i^1 < r_i^0 = a_i.\]
    Every $t \in [0,1]$ is included in exactly one interval $[\ell_i^j,\ell_i^{j+1}]$ or $[r_i^{j+1},r_i^j]$, except the endpoints which are included in two of these intervals, and the points $0$ and $b_i$ are not included in any of these intervals.

    The idea is that $X$ will look somewhat like the following:
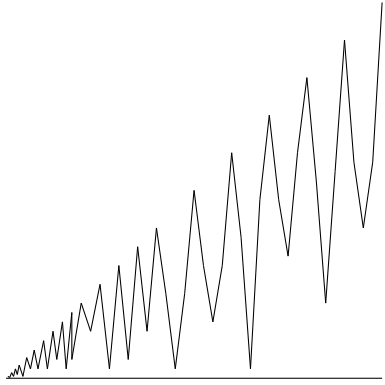
\[\begin{tikzpicture}[scale = 0.5]
\draw  (11.25,6.25) -- (1.25,6.25);
\draw  (11.25,16.25) -- (11,12);
\draw  (11,12) -- (10.75,10.25);
\draw  (10.75,10.25) -- (10.5,12);
\draw  (10.5,12) -- (10.25,15.25);
\draw  (10.25,15.25) -- (10,11.75);
\draw  (10,11.75) -- (9.75,8.25);
\draw  (9.75,8.25) -- (9.5,11.5);
\draw  (9.5,11.5) -- (9.25,14.25);
\draw  (9.25,14.25) -- (9,12.25);
\draw  (9,12.25) -- (8.75,9.5);
\draw  (8.75,9.5) -- (8.5,11);
\draw  (8.5,11) -- (8.25,13.25);
\draw  (8.25,13.25) -- (8,11);
\draw  (8,11) -- (7.75,6.5);
\draw  (7.75,6.5) -- (7.5,10);
\draw  (7.5,10) -- (7.25,12.25);
\draw  (7.25,12.25) -- (7,9.25);
\draw  (7,9.25) -- (6.75,7.75);
\draw  (6.75,7.75) -- (6.5,9.25);
\draw  (6.5,9.25) -- (6.25,11.25);
\draw  (6.25,11.25) -- (6,8.5);
\draw  (6,8.5) -- (5.75,6.5);
\draw  (5.75,6.5) -- (5.5,8.5);
\draw  (5.5,8.5) -- (5.25,10.25);
\draw  (5.25,10.25) -- (5,7.5);
\draw  (5,7.5) -- (4.75,9.75);
\draw  (4.75,9.75) -- (4.5,6.75);
\draw  (4.5,6.75) -- (4.25,9.25);
\draw  (4.25,9.25) -- (4,6.5);
\draw  (4,6.5) -- (3.75,8.75);
\draw  (3.75,8.75) -- (3.5,7.5);
\draw  (3.5,7.5) -- (3.25,8.25);
\draw  (3.25,8.25) -- (3,6.75);
\draw  (3,6.75) -- (3,8);
\draw  (3,8) -- (2.85,6.5);
\draw  (2.85,6.5) -- (2.75,7.75);
\draw  (2.75,7.75) -- (2.6,6.75);
\draw  (2.6,6.75) -- (2.5,7.5);
\draw  (2.5,7.5) -- (2.35,6.5);
\draw  (2.35,6.5) -- (2.25,7.25);
\draw  (2.25,7.25) -- (2.1,6.5);
\draw  (2.1,6.5) -- (2,7);
\draw (2,7) -- (1.9,6.5);
\draw (1.9,6.5) -- (1.8,6.8);
\draw (1.8,6.8) -- (1.7,6.3);
\draw (1.7,6.3) -- (1.6,6.6);
\draw (1.6,6.6) -- (1.55,6.35);
\draw (1.55,6.35) -- (1.5,6.5);
\draw (1.5,6.5) -- (1.45,6.3);
\draw (1.45,6.3) -- (1.4,6.4);
\draw (1.4,6.4) -- (1.35,6.27);
\draw (1.35,6.27) -- (1.3,6.3);
\end{tikzpicture}\]
    The straight horizontal line is $[0,1] \times \{0\}$ and the squiggly line is the graph of $f$. Each of the downward spikes in the graph of $f$ happens between some $a_i$ and $a_{i+1}$, with the points at $b_i$. The larger $|W_i|$, the lower the spike gets, so that if $|W_i| = \infty$, then the spike touches the horizontal line. This makes the space $X$ not homeomorphic to $[0,1]$. If no spike touches the horizontal line, then $X$ will be homeomorphic to $[0,1]$.

    Given $t \in [0,1]$, define $f(t)$ as follows:
    \begin{enumerate}
        \item If $t = 0$, then $f(t) = f(0) = 0$.
        \item If $t = b_i$, then $f(t) = f(b_i) = b_i 2^{-|W_i|}$ (which is equal to $0$ if $W_i$ is infinite).
        \item If $t \in [\ell_i^j,\ell_i^{j+1}]$ then
        \[ f(t) = t \left[ \frac{t-\ell_i^j}{\ell_i^{j+1} - \ell_i^j} \left(2^{-|W_i[j+1]|} - 2^{-|W_i[j]|} \right) + 2^{-|W_i[j]|} \right].\]
        So $f(\ell_i^j) = \ell_i^j 2^{-|W_i[j]|}$ and $f(\ell_i^{j+1}) = \ell_i^{j+1} 2^{-|W_i[j+1]|}$. In between, $f(t)$ is $t$ times a linear interpolation.
        \item If $t \in [r_i^{j+1},r_i^j]$ then
        \[ f(t) =  t \left[ \frac{t-r_i^{j+1}}{r_i^{j} - r_i^{j+1}} \left(2^{-|W_i[j]|} - 2^{-|W_i[j+1]|} \right) + 2^{-|W_i[j+1]|} \right].\]
        This is essentially (3) in reverse. We have $f(r_i^j) = r_i^j 2^{-|W_i[j]|}$ and $f(r_i^{j+1}) = r_i^{j+1} 2^{-|W_i[j+1]|}$.
    \end{enumerate}
    We must check that $f(t)$ is well-defined; that is, for any point which is in two of these intervals, the definitions of $f$ agree on those common points. This is easy to see for the points $\ell_i^j$ with $j > 0$ which are parts of the intervals $[\ell_i^{j-1},\ell_i^j]$ and $[\ell_i^j,\ell_i^{j+1}]$. Similarly, it is easy to see for the points $r_i^j$ with $j > 0$. The non-trivial cases are the points $a_i = r_i^0 = \ell_{i-1}^0$ for $i > 0$. But in this case, $|W_i[0]| = |W_{i-1}[0]| = 0$ and so $f(a_i) = a_i$.

    We also need to check that the function is continuous, which means checking that (1) and (2) agree with the limiting behaviour of $f$. For (1), we have $f(t) \leq t$ and so $\lim_{t \to 0} f(t) = 0$. For (2), for $t \in [\ell_i^j,\ell_i^{j+1}]$ or $t \in [r_i^{j+1},r_i^j]$ we have $t 2^{-|W_i[j+1]|} \leq f(t) \leq t 2^{-|W_i[j]|}$. Thus $f(b_i) = b_i 2^{-|W_i|}$ is the limit as $t \to b_i$.

    We also need to be able to produce $X = X_W$ computably and uniformly from $W$. Note that we cannot compute $|W_i|$ uniformly from $W$, so we cannot compute $f(b_i)$. But we can compute $f$ on all other inputs, and this gives a dense set to produce a presentation of $X_W$.

    Finally, note that if each $W_i$ is finite, then $f(t) > 0$ when $t > 0$, but if $W_i$ is infinite, then $f(b_i) = 0$. As discussed before, $X_W$ is homeomorphic to $[0,1]$ if and only if $f(t) > 0$ for all $t > 0$. This completes the argument.
\end{proof}

\section{\texorpdfstring{$\bfPi^0_4$-hardness for $[0,1]$}{Pi04-hardness for [0,1]}}

We now turn to proving that $\HCopies([0,1])$ is $\bfPi^0_4$-complete.

\subsection{The main idea}

We begin by showing that $\HCopies([0,1])$ is $\bfSigma^0_3$-hard. By attaching infinitely many instances of this construction end-to-end, we will get a $\bfPi^0_4$ completeness result.

Looking at the characterization above, we see that the $\bfPi^0_4$ complexity comes from the local connectedness, and so we should try to make a line in the $\bfSigma^0_3$ outcome, and a non-locally-connected space in the $\bfPi^0_3$ outcome. The non-locally-connected spaces we build will be based off of the following:


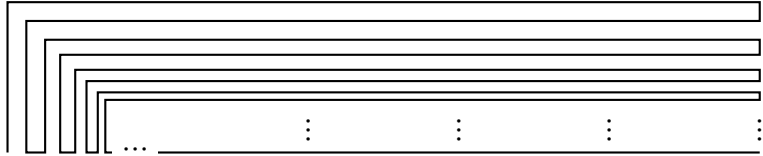
\[\begin{tikzpicture}
		\draw [thick] plot coordinates {(-5,0) (-5,2) (5,2) (5,1.75) (-4.75,1.75) (-4.75,0) (-4.5,0) (-4.5,1.5) (5,1.5) (5,1.3) (-4.3,1.3) (-4.3,0) (-4.1,0) (-4.1,1.1) (5,1.1) (5,0.95) (-3.95,0.95) (-3.95,0) (-3.8,0) (-3.8,0.8) (5,0.8) (5,0.7) (-3.7,0.7) (-3.7,0) (5,0)};
		
		\node at (5,0.3) {$\vdots$};
		\node at (3,0.3) {$\vdots$};
		\node at (1,0.3) {$\vdots$};
		\node at (-1,0.3) {$\vdots$};
		\node [fill=white] at (-3.3,0) {$\cdots$};
\end{tikzpicture}\]
 
 The space is built up from infinitely many \textit{tendrils}. Above we drew each tendril with right-angled turns, like the following:
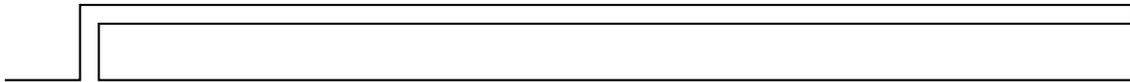
\[\begin{tikzpicture}
    \draw [thick] plot coordinates {(-8,0) (-7,0) (-7,1) (7,1) (7,0.75) (-6.75,0.75) (-6.75,0) 
    (0,0) (7,0)};
\end{tikzpicture}\]
In the actual construction, we will have the tendrils consist of two straight lines, as follows:
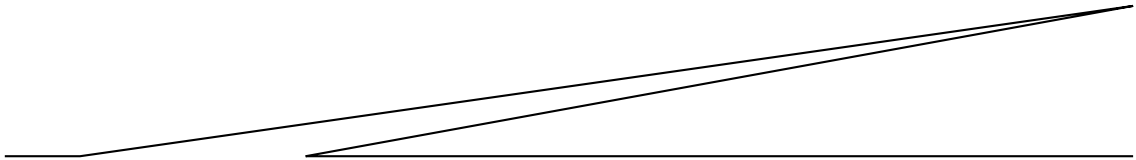
\[\begin{tikzpicture}
    \draw [thick] plot coordinates {(-8,0) (-7,0) (7,2)  (-4,0) 
    (0,0) (7,0)};
\end{tikzpicture}\]
But for clarity of the pictures it will be easier to draw the tendrils using horizontal and vertical lines.
    
Each of the tendrils originates from somewhere on the left half, and goes all the way to the rightmost side. The rightmost point has no small connected neighbourhoods, as each neighbourhood contains some tendril coming from the left side.
 
At each stage $s$, we will have a ``target'' topological space which we are building towards. This target space will be a non-locally-connected space looking something like the space above with many tendrils. However, at each stage of the construction we will have only actually put finitely many points from the target space into the metric space being constructed. For example, suppose that the target space has a tendril
 
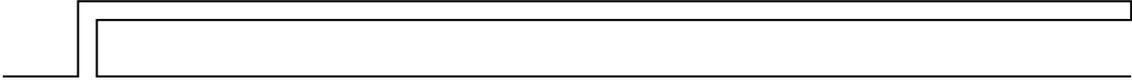
\[\begin{tikzpicture}
    \draw [thick] plot coordinates {(-8,0) (-7,0) (-7,1) (7,1) (7,0.75) (-6.75,0.75) (-6.75,0) 
    (0,0) (7,0)};
 \end{tikzpicture}\]
At a finite stage, we might have only placed points configured as
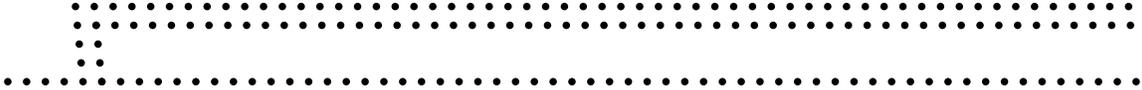
\[\begin{tikzpicture}
    \foreach \x in {-28,...,28}
        \node [fill=black,circle,inner sep=1pt] at (\x*0.25,0) {};

    \foreach \x in {-28,...,28}
        \node [fill=black,circle,inner sep=1pt] at (\x*0.25-0.1,1) {};

    \foreach \x in {-28,...,28}
        \node [fill=black,circle,inner sep=1pt] at (\x*0.25-0.075,0.75) {};
            
    \node [fill=black,circle,inner sep=1pt] at (-7-0.025,0.25) {};
    \node [fill=black,circle,inner sep=1pt] at (-7-0.05,0.5) {};
    
    \node [fill=black,circle,inner sep=1pt] at (-6.75,0) {};

    \node [fill=black,circle,inner sep=1pt] at (-6.75-0.025,0.25) {};

    \node [fill=black,circle,inner sep=1pt] at (-6.75-0.05,0.5) {};

    \node [fill=black,circle,inner sep=1pt] at (-6.75-0.075,0.75) {};

    \node [fill=black,circle,inner sep=1pt] at (-8,0) {};

    \node [fill=black,circle,inner sep=1pt] at (-7.25,0) {};

    \node [fill=black,circle,inner sep=1pt] at (-7.5,0) {};

    \node [fill=black,circle,inner sep=1pt] at (-7.75,0) {};
            
 \end{tikzpicture}\]
We might then decide to \textit{collapse} the tendril, changing the target space to
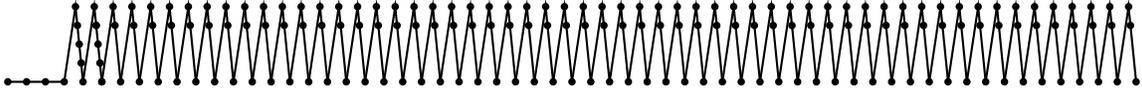
\[\begin{tikzpicture}

    \foreach \x in {-28,...,28}
        \node [fill=black,circle,inner sep=1pt] at (\x*0.25,0) {};

    \foreach \x in {-28,...,28}
        \node [fill=black,circle,inner sep=1pt] at (\x*0.25-0.1,1) {};

    \foreach \x in {-28,...,28}
        \node [fill=black,circle,inner sep=1pt] at (\x*0.25-0.075,0.75) {};
            
    \node [fill=black,circle,inner sep=1pt] at (-7-0.025,0.25) {};
    \node [fill=black,circle,inner sep=1pt] at (-7-0.05,0.5) {};
    
    \node [fill=black,circle,inner sep=1pt] at (-6.75,0) {};

    \node [fill=black,circle,inner sep=1pt] at (-6.75-0.025,0.25) {};

    \node [fill=black,circle,inner sep=1pt] at (-6.75-0.05,0.5) {};

    \node [fill=black,circle,inner sep=1pt] at (-6.75-0.075,0.75) {};

    \node [fill=black,circle,inner sep=1pt] at (-8,0) {};

    \node [fill=black,circle,inner sep=1pt] at (-7.25,0) {};

    \node [fill=black,circle,inner sep=1pt] at (-7.5,0) {};

    \node [fill=black,circle,inner sep=1pt] at (-7.75,0) {};
    
    \foreach \x in {-28,...,28}
        \draw [thick] plot coordinates {(\x*0.25,0) (\x*0.25-0.1,1) (\x*0.25- 0.25,0)};

    \draw [thick] plot coordinates {(-8,0) (-7.25,0) };
 \end{tikzpicture}\]
Since we are already committed to the finitely many points we have placed, and what we have said about the distances between them, they must all stay in the same positions in the new target space. But by changing the target space, we can incorporate the tendril into a (new) homeomorphic copy of the unit interval.

We will then introduce new tendrils, which may themselves get collapsed, after which we introduce more new tendrils, and so on. If we collapse all but finitely many tendrils infinitely often, then in the limit our space will be locally connected and homeomorphic to the unit interval. Otherwise, there will be infinitely many tendrils and our space will not be locally connected.

One remaining technical point is to explain why, in the case that we collapse all but finitely many tendrils, we end up with a space homeomorphic to the unit interval. As illustrated above, we end up with a sawtooth function even after collapsing a single tendril. We need to ensure that in the process of introducing and collapsing new tendrils, this sawtooth is not too wild, and instead comes to a nice limit. To do this, when we introduce new tendrils, they will closely follow the existing sawtooth space. In this way, even though the sawtooth gets more and more bumpy, it comes to a limit.

\subsection{\texorpdfstring{Construction for $\bfSigma^0_3$-completeness}{Construction for Sigma03-completeness}}

Consider the $\Sigma^0_3$ set
\[ A = \{ W \subseteq \omega \times \omega \; \mid \; \exists n \;\; |W_n| = \infty\}\]
where, as before, $W_n$ is the $n$th column of $W$.

Given $W$, we must build a Polish space $X_W = \overline{M_W}$ such that $W \in A$ if and only if $X_W$ is homeomorphic to $[0,1]$. The main idea is that whenever an element enters $W_{n}$ we will collapse the $n$th tendril along with every greater tendril. Thus if for every $n$ the set $W_{n}$ is finite then for each $n$ there is some point after which we stop collapsing the $n$th tendril. In that case, we have infinitely many tendrils, and $X_W$ is not homeomorphic to $[0,1]$. Otherwise, if $W_{n}$ is infinite for some $n$, then we collapse the $n$th tendril and every greater tendril infinitely often, and in the end have only finitely many tendrils. So $X_W$ is homeomorphic to $[0,1]$.

For the following, fix the $W$ for which we will construct $X_W$. For simplicity, we suppress the subscript $W$, writing, e.g., $X$ for $X_W$ and $M$ for $M_W$. The construction will be a dynamic construction, and as before we use $W_n[s] = W_n \res s$ for the restriction of $W_n$ to elements $< s$ which have appeared by stage $s$. We can view $W_n[s]$ as the portion of $W_n$ that we have seen by stage $s$, just like if we were doing a many-one reduction in computability theory. (For the computability theorist, one can think of $W_n$ as the set of stages at which a new element enters a c.e.\ set.) We assume that at each stage $s$, at most one new element enters among all of the sets $W_n$.

We will work inside $\mathbb{R}^2$ and indeed inside $[-1,1]\times[0,2]$; we call these the $x$ and $y$-coordinates, respectively. To simplify a later part of the construction, we will ensure that there is only one special point for each particular value of the $x$ coordinate. Each special point of $M$ will be a rational tuple.

\subsubsection{The data}

At each stage $s$, we will have the following data:
\begin{enumerate}
    \item The \textit{main line}: A function $\mathfrak{m} : [-1,1] \to [0,2]$ with $\mathfrak{m}(-1) = 0$ and $\mathfrak{m}(1) = 0$.
    \item Infinitely many \textit{tendrils}, indexed by $n \in \omega$, where the $n$th tendril consists of the data:
    \begin{enumerate}
        \item a \textit{left attachment point} $\ell_n$, given as an $x$-value;
        \item a \textit{right attachment point} $r_n$, given as an $x$-value;
        \item a \textit{top arc} $\mathfrak{t}_n \colon [\ell_n,1] \to [0,2]$;
        \item a \textit{bottom arc} $\mathfrak{b}_n \colon [r_n,1] \to [0,2]$.
    \end{enumerate}
\end{enumerate}
The functions $\mathfrak{m}$, $\mathfrak{t}_n$, and $\mathfrak{b}_n$ should be thought of as representing the points on their graphs, e.g., the actual points in $\mathbb{R}^2$ making up the $n$th tendril will be the union of the graphs of $\mathfrak{t}_n$ and $\mathfrak{b}_n$. These functions will change stage-by-stage, and for the value at stage $s$ we will write $\mathfrak{m}[s]$, $\mathfrak{t}_n[s]$, and $\mathfrak{b}_n[s]$. When the stage is understood, this may be omitted.

The data should satisfy the following conditions:
\begin{enumerate}
    \item Each tendril is attached to the main line on the left side, and is above the main line. For the $n$th tendril, we have:
    \begin{enumerate}
        \item $\ell_n < r_n < 0$: the left attachment point is to the left of the right attachment point, and both are to the left of the midpoint.
        \item $\mathfrak{m}(\ell_n) = \mathfrak{t}_n(\ell_n)$: the top arc of the tendril attaches to the main line at the left attachment point.
        \item $\mathfrak{m}(r_n) = \mathfrak{b}_n(r_n)$: the bottom arc of the tendril attaches to the main line at the right attachment point.
        \item $\mathfrak{t}_n(1) = \mathfrak{b}_n(1)$: the top and bottom arcs of the tendril meet at $x = 1$.
        \item For $x \in (\ell_n,r_n]$ we have $\mathfrak{m}(x) < \mathfrak{t}_n(x)$, and for $x \in (r_n,1)$ we have $\mathfrak{m}(x) < \mathfrak{b}_n(x) < \mathfrak{t}_n(x)$.
    \end{enumerate}
    \item The tendrils are attached to the main line in order from left to right, and are non-intersecting:
    \begin{enumerate}
        \item $\ell_0 < r_0 < \ell_1 < r_1 < \cdots$.
        \item for $m < n$ and values where they are both defined, i.e., $x \in (\ell_n,1)$, we have $\mathfrak{t}_n(x) < \mathfrak{b}_m(x)$.
    \end{enumerate}
\end{enumerate}
Given the above data, the \textit{target space} is the closed subspace which consists of the main line with the portion between $\ell_n$ and $r_n$ replaced by the $n$th tendril. That is, the following points:
\begin{enumerate}
    \item the points $(x,\mathfrak{m}(x))$ for $x \notin (\ell_n,r_n)$ for any $n$.
    \item for each $n$,
    \begin{enumerate}
        \item the points $(x,\mathfrak{t}_n(x))$ for $x \in [\ell_n,1]$;
        \item the points $(x,\mathfrak{b}_n(x))$ for $x \in [r_n,1]$.
    \end{enumerate}
\end{enumerate}
Recall that the target space is \textit{not} homeomorphic to $[0,1]$, as it is the target in $\Pi^0_3$ outcome. (The fact that the space is closed requires also that the tendrils converge to the main line as $n \to \infty$, which will be the case.)

\subsubsection{The construction at each stage}

Recall that we are building a countable metric space $M = M_W$ as a subspace of $\mathbb{R}^2$ with the induced metric. We begin at stage $0$ with $M$ being almost empty, and at each stage we may put some points into $M$. Then the topological space $X = X_W$ will be the completion of $M$.

Let $\{2,p_{n,c},q_{n,c} : n,c \in \omega\}$ be a computable listing of all the primes with no repetitions.

\medskip

\noindent \textit{Construction at stage $0$:} We begin the construction at stage $0$ with the main line $\mathfrak{m}(x) = 0$ and the tendrils given by $\ell_n = -\frac{1}{2}-\frac{1}{2n+2}$, $r_n = -\frac{1}{2}-\frac{1}{2n+3}$, and
\begin{align*}
    \mathfrak{t}_n: [\ell_n,1] &\to [0,2] \\
    x &\mapsto \mathfrak{m}(x) + \frac{1}{2^n} \left(\frac{x-\ell_n}{1-\ell_n}\right)
\end{align*}
and
\begin{align*}
    \mathfrak{b}_n: [r_n,1] &\to [0,2] \\
    x &\mapsto \mathfrak{m}(x) + \frac{1}{2^n} \left(\frac{x-r_n}{1-r_n} \right)
\end{align*}
Put $(-1,0)$ and $(1,0)$ into the metric space $M = M_W$.

\medskip

\noindent \textit{Construction at stage $s+1$:} Suppose that at stage $s+1$, a new element enters $W_{n}[s+1]$ that was not in $W_n[s]$. Then:
\begin{itemize}
    \item First, we collapse all tendrils numbered $n$ onward. We must incorporate the special points on these tendrils that we have already included in $M$ into the main line. We define the new main line $\mathfrak{m}[s+1]$ as follows. For $x \in (-1,\ell_n]$, we keep the main line the same: $\mathfrak{m}[s+1](x) := \mathfrak{m}[s](x)$. Now consider all of the special points $(x,y)$ which we have already put in $M$ which are either on the main line (with $x \in [\ell_n,1]$) or on the $n$th tendril or a greater one. If $n>0$, these are precisely the points $(x,y)$ with $x \in [\ell_n,1]$ and $y < \mathfrak{b}_{n-1}[s](x)$. If $n=0$, these are all of the points already in $M$ with $x \in [\ell_0,1]$. Order these points from left to right as $(x_1,y_1),(x_2,y_2),\ldots,(x_k,y_k)$. Recall that there is no repetition among the $x_i$, as no two special points share the same $x$-coordinate. Let $\mathfrak l\colon[\ell_n,1]\to[0,2]$ be a rational piecewise-linear function whose graph contains the points
    \[
    (\ell_n,\mathfrak m[s](\ell_n)),(x_1,y_1),\ldots,(x_k,y_k),(1,0),
    \]
    and such that, for every $x\in[\ell_n,1]$, $\mathfrak m[s](x)\leq\mathfrak l(x) \leq\mathfrak t_n[s](x)$. We will verify later that the listed points are also bounded by these two functions. Let $\mathfrak m[s+1](x)=\mathfrak l(x)$ for $x\in[\ell_n,1]$ (and leave it unchanged otherwise). Note that $(-1,0)$ and $(1,0)$ were put into $M$ at stage $0$, and so we have $\mathfrak m[s+1](-1)=\mathfrak m[s+1](1)=0$.

    \item Second, we must create new tendrils. Choose new values $\ell_n, r_n, \ell_{n+1}, r_{n+1},\ldots$. If $n>0$, choose them such that
    \[ r_{n-1} < -\frac {1}{2s+2} < \ell_n < r_n < \ell_{n+1} < r_{n+1} < \cdots < -\frac{1}{2s+4} < 0. \]
    If $n=0$, choose them such that
    \[ -\frac {1}{2s+2} < \ell_0 < r_0 < \ell_1 < r_1 < \cdots < -\frac{1}{2s+4} < 0. \]
    In either case, require that there are no special points between $\ell_n$ and $r_n$, between $\ell_{n+1}$ and $r_{n+1}$, and so on. Let
    \[ g := \inf_{x \in [\ell_n,1]} \left(\mathfrak{b}_{n-1}(x) - \mathfrak{m}[s+1](x)\right).\]
    Think of $1 \geq g > 0$ as the gap between the new main line $\mathfrak{m}$ and the remaining tendrils above $\mathfrak{m}$ (which have not been collapsed).\footnote{If $n = 0$, we modify this appropriately with $g = 1$.}
    Then, for every $k \geq n$, define
    \begin{align*}
        \mathfrak{t}_k: [\ell_k,1] &\to [0,2] \\
        x &\mapsto \mathfrak{m}[s+1](x) + \frac{g}{2^{k+s+1}} \left(\frac{x-\ell_k}{1-\ell_k}\right)
    \end{align*}
    and
    \begin{align*}
        \mathfrak{b}_k: [r_k,1] &\to [0,2] \\
        x &\mapsto \mathfrak{m}[s+1](x) + \frac{g}{2^{k+s+1}} \left(\frac{x-r_k}{1-r_k} \right).
    \end{align*}
    It is somewhat tedious but not hard to see that the new main line and tendrils satisfy all of the conditions required.
\end{itemize}
If no element enters any $W_n$ at stage $s+1$, then we leave everything the same (e.g., $\mathfrak{m}[s+1] = \mathfrak{m}[s]$).

We may now have a new target space (though if no elements entered any $W_n$, then our target space remains the same). We must add new points to our space. For each $n$, let $c_n$ be the number of times that the $n$th tendril has been collapsed. Add to $M$ the following points from the target space:
\begin{enumerate}
    \item for each $x \in (-1,1)$, $x \neq 0$, of the form $\pm t/2^{s+1}$ for any $t \in \mathbb{N}$, and such that $x \notin (\ell_n,r_n)$ for any $n$, add $(x,\mathfrak{m}(x))$ to $M$.
    \item for each $n<s$ and each $x \in [\ell_n,1)$, $x \neq 0$, of the form $\pm t/p_{n,c_n}^{s+1}$ for any $t \in \mathbb{N}$, add $(x,\mathfrak{t}_n(x))$ to $M$.
    \item for each $n<s$ and each $x \in [r_n,1)$, $x \neq 0$, of the form $\pm t/q_{n,c_n}^{s+1}$ for any $t \in \mathbb{N}$, add $(x,\mathfrak{b}_n(x))$ to $M$.
\end{enumerate}
Note that $p_{n,c_n}$ and $q_{n,c_n}$ are unique primes not only for the $n$th tendril, but unique for this instance of the $n$th tendril (i.e., different than the primes used for previous instances of the $n$th tendril that have now been collapsed). Also, for each point $(x,y)$ that we have added to $M$, $y \geq \mathfrak{m}[s+1](x)$. Thus the main line is a lower bound on all of the points in $M$.

\medskip

This completes the construction.

\subsection{Verification}

Fix $W$ for which we have constructed $X = X_W$. We begin with a list of properties of the construction.
\begin{remark}\label{rem:verifyconstruct}
We have the following:
\begin{enumerate}
    \item If at stage $s$ we collapse the $n$th, $n+1$st, \ldots, tendril, then $\mathfrak{t}_n[s-1] \geq \mathfrak{l}$ where $\mathfrak{l}$ is as defined at stage $s$.

    \item At any stage $s$, the first tendril is above the second tendril, the second is above the third, and so on. All of the tendrils are above the main line.

    Both points are easy to see by following the construction. The point of the gap $g$ is to ensure that this is the case; also, note that the line $\mathfrak{l}$ is under any of the tendrils that are not collapsed. The bound $g\leq 1$ follows inductively. Initially, every tendril lies at most $1$ above the main line. At a later action with $n>0$, the gap defining $g$ is bounded by the gap between $\mathfrak{b}_{n-1}$ and the main line, which is at most $1$ by induction; when $n=0$, we set $g=1$. Raising the main line can only decrease the gaps below unchanged tendrils, and the gap below each newly defined tendril is at most $g/2^{k+s+1}\leq 1$. The upper bound $2$ on the functions is preserved as well. At stage $0$, all the functions take values in $[0,1]$. If $n>0$, the new tendrils lie below $\mathfrak{b}_{n-1}$. If the zeroth tendril is collapsed at stage $s+1$, the new main line lies below the previous zeroth tendril and the new zeroth tendril lies at most $2^{-(s+1)}$ above the new main line. Hence every main line and tendril takes values in
    \[ \left[0,1+\sum_{s=0}^{\infty}2^{-(s+1)}\right]=[0,2]. \]

    \item If a special point $p$ is put in $M$ at stage $s$, then $p$ is above the main line $\mathfrak{m}[s]$.

    This is because all of the tendrils are above the main line, and the points that are added to $M$ are either on a collapsed tendril or on the main line.
    
    \item For each fixed $x$, the sequence $\mathfrak{m}[s](x)$ is increasing.

    We define $\mathfrak{m}[s+1](x) = \max(\mathfrak{m}[s](x),\mathfrak{l}(x))$ or $\mathfrak{m}[s+1](x) = \mathfrak{m}[s](x)$ depending on what happens at stage $s+1$.

    \item If, at a stage $s$, a special point $p = (x,y)$ is in $M$ and is on the main line, then it is on the main line at any greater stage. Moreover, from stage $s$ on, $x$ is never between any pair of left and right endpoints $\ell_i$ and $r_i$.

    At stage $s$, suppose that we collapse the $n$th, $n+1$st, \ldots, tendrils. If $p = (x,y)$ with $x \in [-1,\ell_n]$ then $\mathfrak{m}[s+1](x) = \mathfrak{m}[s](x) = y$ and so $p$ is still on the main line. Otherwise, if $x \in [\ell_n,1]$, then $p$ is on the line $\mathfrak{l}$ defined at stage $s$, and $\mathfrak{m}[s+1](x) = \max(\mathfrak{m}[s](x),\mathfrak{l}(x))$ but $\mathfrak{m}[s](x) = \mathfrak{l}(x) = y$.

    For the moreover clause, note that if $p$ is added directly onto the main line, then it is not between any pair of left and right endpoints. If on the other hand $p$ is put into $M$ as part of a tendril and becomes part of the main line due to collapse, then any pair of left and right endpoints containing it are redefined. Finally, in either case, when any pair of left and right endpoints are redefined at any later stage, they do not contain between them any special points of $M$.

    \item If the $n$th tendril was first defined at stage $t$, and was never collapsed between stage $t$ and stage $s > t$, then $\mathfrak{t}_n(x) - \mathfrak{m}[s](x) \leq 2^{-n-t}$.

    At stage $t$, we would have defined $\mathfrak{t}_n(x) = \mathfrak{t}_n[t](x)$ such that $\mathfrak{t}_n(x) - \mathfrak{m}[t](x) \leq 2^{-n-t}$. But $\mathfrak{m}[t](x) \leq \mathfrak{m}[s](x) \leq \mathfrak{t}_n(x)$.


    \item If, at stage $s+1$, a new element enters $W_n$ for the $k$th time, where $k\geq 1$, then $\Vert \mathfrak{m}[s] - \mathfrak{m}[s+1] \Vert \leq 2^{-n-k+1}$.

    Let $u$ be the stage at which the current instance of the $n$th tendril was most recently defined. If $k=1$, then $u\geq 0$, and the defining formula together with Remark \ref{rem:verifyconstruct} (6) bounds the height of this tendril above the current main line by $2^{-n-u}\leq 2^{-n}=2^{-n-k+1}$. If $k>1$, then this tendril was most recently defined no earlier than the stage of the $(k-1)$st entry into $W_n$. That stage is at least $k-1$, so $u\geq k-1$ and the same bound gives $2^{-n-u}\leq 2^{-n-k+1}$. During the collapse, $\mathfrak{m}[s](x)\leq\mathfrak{m}[s+1](x)\leq\mathfrak{t}_n[s](x)$ for $x\in[\ell_n,1]$, while $\mathfrak{m}[s+1](x)=\mathfrak{m}[s](x)$ elsewhere. This gives the stated norm bound.

    \item If $\ell_n$ or $r_n$ was defined at stage $s$, then $\ell_n < r_n < -\frac{1}{2s+2}$.
    
    This is because at stage $0$ we initially define $\ell_n < r_n < -\frac{1}{2}$ and if, at stage $s+1$ we redefine them, it is with $-\frac {1}{2s+2} < \ell_n < r_n < \ell_{n+1} < r_{n+1} < \cdots < -\frac{1}{2s+4} < 0$.

\end{enumerate}
\end{remark}

First, we will argue that the main line comes to a limit $\mathfrak{m}_\infty$.

\begin{lemma}
    The limit $\mathfrak{m}_{\infty}(x) = \lim_{s \to \infty} \mathfrak{m}[s](x)$ exists, and $\mathfrak{m}_{\infty} \colon [-1,1]\to[0,2]$ is continuous. $\mathfrak{m}_\infty(-1) = \mathfrak{m}_\infty(1) = 0$.
\end{lemma}
\begin{proof}
    For a fixed $x$, the sequence $(\mathfrak{m}[s](x))_{s \in \mathbb{N}}$ is increasing and bounded above (by $2$), and hence converges. Thus $\mathfrak{m}[s]$ converges pointwise to a function $\mathfrak{m}_\infty \colon [-1,1] \to [0,2]$.

    To prove that the convergence is uniform, if $s<t$, then by Remark \ref{rem:verifyconstruct} (7),
    \[
    \Vert\mathfrak{m}[t]-\mathfrak{m}[s]\Vert
    \leq \sum_{(n,k)}2^{-n-k+1},
    \]
    where the sum can be taken over only those $(n,k)$ for which the $k$th element entered $W_n$ between stages $s$ and $t$. The sum over all $(n,k)$ converges:
    \[
    \sum_{n=0}^{\infty}\sum_{k=1}^{\infty}2^{-n-k+1}=4.
    \]
    Since each pair $(n,k)$ can occur only once, the sums over action pairs occurring between stages $s$ and $t$ are tails of this convergent series. Thus, for $s<t$ with $s$ sufficiently large, $\Vert\mathfrak{m}[t]-\mathfrak{m}[s]\Vert$ is arbitrarily small. Hence $(\mathfrak{m}[s])$ is Cauchy under the uniform metric and converges uniformly to its pointwise limit $\mathfrak{m}_\infty$, which is therefore continuous. The final claim of the lemma follows because for each $s$, $\mathfrak{m}[s](-1) = \mathfrak{m}[s](1) = 0$.
\end{proof}

\subsubsection{\texorpdfstring{The case $W \in A$}{The case W in A}}

Suppose that $W \in A$. We must argue that $X_W$ is homeomorphic to $[0,1]$. Let $n$ be least such that $W_n$ is infinite, and let $s^*$ be a stage after which no element ever enters any $W_i$ with $i<n$.

For each $i<n$, let $\ell_i,r_i,\mathfrak{b}_i,\mathfrak{t}_i$ be the data of the $i$th tendril at stage $s^*$. The tendril is never collapsed at any later stage, so this is the final data associated with the tendril. Let $Y$ consist of:
\begin{enumerate}
    \item the points $(x,\mathfrak{m}_\infty(x))$ for $x \in [-1,1] \setminus \bigcup_{i<n}(\ell_i,r_i)$;
    \item for each $i < n$,
    \begin{enumerate}
        \item the points $(x,\mathfrak{t}_i(x))$ for $x \in [\ell_i,1]$;
        \item the points $(x,\mathfrak{b}_i(x))$ for $x \in [r_i,1]$.
    \end{enumerate}
\end{enumerate}
Each connected component of the main-line portion, and each of the remaining graphs listed above, is the graph of a continuous injective function over a closed interval and hence homeomorphic to $[0,1]$. Thus $Y$ is the union of finitely many closed arcs, which are connected end-to-end and otherwise disjoint, and so $Y$ is homeomorphic to $[0,1]$. Now we must argue that $X = X_W$ is equal to $Y$, that is, that $M = M_W$ is a subset of $Y$ and that $\overline{M} = Y$. 

First we argue that $M \subseteq Y$. Suppose that a point $(x,y)$ is put into $M$ at stage $s$. There are three possible ways that it might be put into $M$. If it is put into the main line, then by Remark \ref{rem:verifyconstruct} (5) $(x,y)$ will be part of the main line at every later stage (and will never be between any pair of left and right endpoints). If $(x,y)$ is added to $M$ as part of the $i$th tendril, then either that tendril is never later collapsed, in which case it remains in that tendril, or the tendril is at some later point collapsed to the main line, and $(x,y)$ remains in the main line at all stages after that.

Now we will argue that the closure $X = \overline{M}$ covers all of $Y$. Given $(x,y) \in Y$ and $\epsilon > 0$, we break into cases. Suppose that $(x,y)$ is a point of the form $(x,\mathfrak{m}_\infty(x))$ on the main line, with $x \notin (\ell_i,r_i)$ for any $i<n$. Choose $\delta > 0$ and a stage $s$ such that
\begin{enumerate}
    \item if $|x-x'| < \delta$ then $|\mathfrak{m}_\infty(x)-\mathfrak{m}_\infty(x')| < \epsilon/3$;
    \item at stage $s$, either $(x-\delta,x] $ or $[x,x+\delta)$ is disjoint from any interval $(\ell_i,r_i)$;
    \item $2^{-s} < \delta$;
    \item $2^{-s} < \epsilon/3$;
    \item $\Vert \mathfrak{m}_\infty - \mathfrak{m}[s] \Vert < \epsilon / 3$. 
\end{enumerate}
It takes some argument to see that we can do this. We can choose $\delta$ satisfying (1) since $\mathfrak{m}_\infty$ is uniformly continuous (as it is continuous on a compact set). Moreover, (1) is maintained if we shrink $\delta$. We obtain (2) as follows. If $x \geq 0$, then (2) is clear for $[x,x+\delta)$. Otherwise, if $x < 0$, we argue as follows. By shrinking $\delta$, we may assume that for any $s \geq s^*$ (2) holds restricted solely to $i<n$. By taking $s$ sufficiently large that $x+\delta < -\frac{1}{2s}$, and moreover taking $s$ to be a sufficiently large stage such that an element enters $W_n$ at stage $s$, by Remark \ref{rem:verifyconstruct} (8) at stage $s$ we have that $x + \delta < -\frac{1}{2s} < \ell_n < r_n < \cdots$ and moreover this remains true at any stage $s' \geq s$. Increasing $s$ further, we can obtain (3), (4), and (5).

Then, at the chosen stage $s$, there is a nonzero point of the form $x' = \pm t/2^s$, with $t\in\mathbb{N}$, within $2^{-s}$ of $x$ and not in any interval $(\ell_i,r_i)$. Then $(x',\mathfrak{m}[s](x'))$ is added to $M$ at stage $s$. Since $|x-x'| < 2^{-s} < \delta$, $|\mathfrak{m}_\infty(x)-\mathfrak{m}_\infty(x')| < \epsilon/3$. Also, $|\mathfrak{m}_\infty(x')-\mathfrak{m}[s](x')| < \epsilon/3$. Thus the distance between $(x,\mathfrak{m}_\infty(x))$ and $(x',\mathfrak{m}[s](x'))$ is at most \[|x-x'|+|\mathfrak{m}_\infty(x)-\mathfrak{m}[s](x')| \leq |x-x'|+|\mathfrak{m}_\infty(x)-\mathfrak{m}_\infty(x')|+|\mathfrak{m}_\infty(x')-\mathfrak{m}[s](x')| < \epsilon/3+\epsilon/3+\epsilon/3 = \epsilon.\]

If $(x,y)$ is a point of the form $(x,\mathfrak{t}_i(x))$ or $(x,\mathfrak{b}_i(x))$, $i < n$, a simpler version of the above argument works, where the stage $s$ is taken to be large enough that the $i$th tendril has stabilized. Here the sampling points $\pm t/2^s$ and mesh size $2^{-s}$ are replaced by $\pm t/p_{i,c_i}^s$ and $p_{i,c_i}^{-s}$, or by $\pm t/q_{i,c_i}^s$ and $q_{i,c_i}^{-s}$, respectively.

\begin{remark}
    The fact that $\mathfrak{m}(-1)=\mathfrak{m}(1)=0$ means that the homeomorphism $[0,1] \to X$ maps $0$ to $(-1,0)$ and $1$ to $(1,0)$.
\end{remark}

\subsubsection{\texorpdfstring{The case $W \notin A$}{The case W not in A}}

Suppose that $W \notin A$, so that for each $n$, $W_n$ is finite. We will argue that $X_W$ is not locally connected, whence it is not homeomorphic to $[0,1]$.

For each $n$, there is a stage $s_n$ after which no element enters any $W_i$ with $i\leq n$. Thus, after stage $s_n$, the $n$th tendril remains the same. It is easy to see that the $n$th tendril---that is, the points $(x,\mathfrak{t}_n(x))$ for $x \in [\ell_n,1]$ and $(x,\mathfrak{b}_n(x))$ for $x \in [r_n,1]$---are in $X = X_W$. There are infinitely many tendrils and we argue that $X$ is not locally connected.

Indeed, let $B$ be any open set containing $(1,0)$ and of diameter at most, say, $1/2$. Then for sufficiently large $n$, say $n \geq N$, the tip of the $n$th tendril, $(1,\mathfrak{t}_n(1)) = (1,\mathfrak{b}_n(1))$ is contained within $B$. However the left and right endpoints of the tendril lie outside of $B$ since the diameter of $B$ is at most $1/2$. The tip of each tendril ($n \geq N$) is in a different connected component of $B$. Thus $B$ is not connected.

\subsection{\texorpdfstring{The construction for $\bfPi^0_4$-completeness}{The construction for Pi04-completeness}}\label{sec:pi4}

Let $B$ be the $\bfPi^0_4$-Wadge-complete set
\[ B = \{ U \subseteq \omega \times \omega \times \omega \; \mid \; \forall m \;\; U_m \in A\}.\]
Given $U$, for each $m$, let $X_{U_m}$ be the metric space produced by the construction just given when applied to $U_m$; that is, $X_{U_m}$ is not locally connected if $U_m \notin A$, and $X_{U_m}$ is homeomorphic to $[0,1]$ if $U_m \in A$. Moreover, in the latter case, the homeomorphism maps $0 \mapsto (-1,0)$ and $1 \mapsto (1,0)$. By rescaling and translating only the $x$-coordinate, we may assume that $X_{U_m} \subseteq [0,1]\times[0,2]$ with $0 \mapsto (0,0)$ and $1 \mapsto (1,0)$.

Let $Y_U$ be the space which contains, for each $m$, a copy of $X_{U_m}$ which is scaled down and shifted so that the point corresponding to $(1,0)$ in $X_{U_m}$ matches up with the point corresponding to $(0,0)$ in $X_{U_{m+1}}$. The $m$th copy has width $1/2^{m+1}$ and height $1/2^{m-1}$. Moreover, we ``twist'' each subsequent copy in a new direction so that there is no other intersection. Without the twists, this looks like:
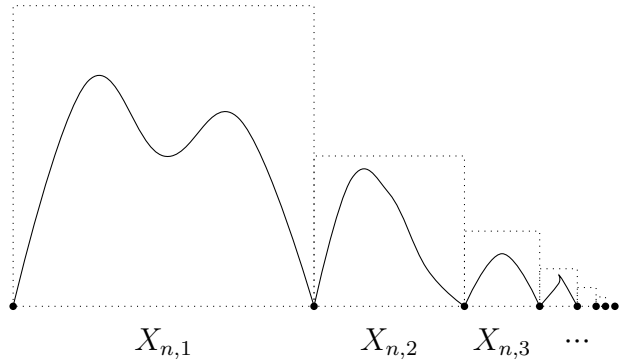
\[ \begin{tikzpicture}
    \draw [dotted] plot coordinates {(-8,0) (-8,4) (-4,4) (-4,0) (-8,0)};
    \draw [dotted] plot coordinates {(-4,0) (-4,2) (-2,2) (-2,0) (-4,0)};
    \draw [dotted] plot coordinates {(-2,0) (-2,1) (-1,1) (-1,0) (-2,0)};
    \draw [dotted] plot coordinates {(-1,0) (-1,0.5) (-0.5,0.5) (-0.5,0) (-1,0)};
    \draw [dotted] plot coordinates {(-0.5,0) (-0.5,0.25) (-0.25,0.25) (-0.25,0) (-0.5,0)};
    \draw [dotted] plot coordinates {(-0.25,0) (-0.25,0.125) (-0.125,0.125) (-0.125,0) (-0.25,0)};

    \node [fill=black,circle,inner sep=1pt] at (-8,0) {};
    \node [fill=black,circle,inner sep=1pt] at (-4,0) {};
    \node [fill=black,circle,inner sep=1pt] at (-2,0) {};
    \node [fill=black,circle,inner sep=1pt] at (-1,0) {};
    \node [fill=black,circle,inner sep=1pt] at (-0.5,0) {};
    \node [fill=black,circle,inner sep=1pt] at (-0.25,0) {};
    \node [fill=black,circle,inner sep=1pt] at (-0.125,0) {};
    \node [fill=black,circle,inner sep=1pt] at (0,0) {};

    \draw [] plot [smooth, tension=.7] coordinates {(-8,0)(-7,3) (-6,2) (-5,2.5) (-4,0)};
    \draw [] plot [smooth, tension=.7] coordinates {(-4,0)(-3.5,1.7) (-3,1.5) (-2.5,0.5) (-2,0)};
    \draw [] plot [smooth, tension=.7] coordinates {(-2,0)(-1.5,0.7) (-1,0)};
    \draw [] plot [smooth, tension=.7] coordinates {(-1,0)(-0.75,0.3) (-0.73,0.4)  (-0.5,0)};

    \node at (-6,-0.5) {$X_{U_0}$};
    \node at (-3,-0.5) {$X_{U_1}$};
    \node at (-1.5,-0.5) {$X_{U_2}$};

    \node at (-0.5,-0.5) {$\cdots$};
\end{tikzpicture}\]
We work in $c_0$, the space of sequences in $\mathbb{R}^{\mathbb{N}}$ which converge to $0$. Let $\mathbf{u}$ be $(1,0,0,0,\ldots)$ and let $\mathbf{w}_0,\mathbf{w}_1,\ldots$ be a list of the other standard basis elements of $c_0$. These $\mathbf{w}_i$ will be the direction for the twists. For each $m\geq 0$, $Y_U$ will contain the following points:
\[ \left\{ \left(\sum_{j=0}^{m-1}\frac{1}{2^{j+1}} + \frac{1}{2^{m+1}} x\right) \mathbf{u} + \frac{1}{2^m}y \mathbf{w}_m : (x,y) \in X_{U_m}\right\} \]
$Y_U$ will also contain $\mathbf{u}$ as the limit point of the above points. It is not hard to see that $Y_U$ is homeomorphic to $[0,1]$ if and only if each $X_{U_m}$ is homeomorphic to $[0,1]$. Thus $Y_U$ is homeomorphic to $[0,1]$ if and only if $U \in B$, as desired.

\begin{remark}\label{rem:startend}
    In Theorem \ref{thm:circle-complete} we will use the fact that $Y_U$ always includes $(0,0,0,\ldots)$ and $(1,0,0,\ldots)$ and, in the case that it is homeomorphic to $[0,1]$, then these are the two endpoints. This will allow us to easily glue multiple copies of $Y_U$ next to each other.
\end{remark}

\section{\texorpdfstring{The circle $S^1$}{The circle}}

We now show that $\HCopies(S^1)$ is $\Pi^0_4$. The strategy is similar to that for $[0,1]$, and we reuse conditions such as $\compact$ and $\conn$ from Section \ref{sec:three}. Our new condition describes the circle by saying that any four points can be arranged in order around the circle, where we use the same ideas of betweenness as before to identify the order in which they are arranged.

\begin{definition}
     $X = \overline{(M,d)}$ satisfies \twodisconn if for any distinct special points $x_0,x_1,x_2,x_3$ (where we consider $0,1,2,3 \in \mathbb{Z} / 4 \mathbb{Z}$) there is some reordering of them such that
     \begin{enumerate}
         \item for all $i \in \mathbb{Z}/4\mathbb{Z}$ and $\delta > 0$ there is $\rho < \delta$ such that for all $\epsilon$ there is an $\epsilon$-path $u_1 = x_i,\ldots,u_n = x_{i+1}$ with no $u_j$ in $\overline{B}_\rho(x_{i+2}) \cup \overline{B}_\rho(x_{i+3})$.\footnote{Recall that we use $\overline{B}_\rho(x)$ for the closed ball rather than the closure of the open ball.}
         \item for all $i \in \mathbb{Z}/4\mathbb{Z}$ and $\delta > 0$ there is $\epsilon$ such that for any $\epsilon$-path $u_1 = x_i,\ldots,u_n = x_{i+2}$, for some $j$ either $u_j \in B_\delta(x_{i+1})$ or $u_j \in B_\delta(x_{i-1})$.
     \end{enumerate}
\end{definition}

Note that this is $\Pi^0_4$.

\begin{lemma}
    If $X = \overline{(M,d)}$ is homeomorphic to $S^1$, then it satisfies \twodisconn (as well as \nondegen, \compact, \conn, and \localconn).
\end{lemma}
\begin{proof}
    All of these properties other than \twodisconn follow from Section \ref{sec:three}. So we now verify \twodisconn. Given any four points on $S^1$, list them in order around the circle as $x_0,x_1,x_2,x_3$.

    For (1), given $i$ and $\delta > 0$, choose $\rho < \delta$ as follows. Let $C$ be the segment of $S^1$ between $x_i$ and $x_{i+1}$ not containing $x_{i+2}$ and $x_{i+3}$. Choose $\rho < \delta$ such that $2 \rho$ is smaller than the distance between this segment and either of $x_{i+2}$ or $x_{i+3}$. Then for any $\epsilon$, we can choose an $\epsilon$-path approximating the arc $C$ between $x_i$ and $x_{i+1}$ (in the sense that each element of the $\epsilon$-path is within $\rho$ of $C$).

    For (2), given $i$ and $\delta > 0$, if either $x_i$ or $x_{i+2}$ are in $B_\delta(x_{i-1})$ or $B_\delta(x_{i+1})$, then (2) holds trivially. Otherwise, $X - B_\delta(x_{i+1}) - B_\delta(x_{i-1})$ is closed and disconnected, and $x_i$ and $x_{i+2}$ are in different connected components. Choosing $\epsilon$ to be smaller than the distance between clopen (in $X - B_\delta(x_{i+1}) - B_\delta(x_{i-1})$) separating sets containing each of them, there is no $\epsilon$-path from $x_i$ to $x_{i+2}$ in $X - B_\delta(x_{i+1}) - B_\delta(x_{i-1})$.
\end{proof}

\begin{lemma}\label{lem:circle}
    If $X = \overline{(M,d)}$ satisfies \compact, \conn, \localconn, and \twodisconn then for any four pairwise distinct special points, there is some ordering $x_0,x_1,x_2,x_3$ and arcs $x_0 \to x_1 \to x_2 \to x_3 \to x_0$ homeomorphic to $S^1$. Moreover, any path from $x_i$ to $x_{i+2}$ must pass through $x_{i+1}$ or $x_{i-1}$.
\end{lemma}
\begin{proof}
    Since $X = \overline{(M,d)}$ satisfies \compact, \conn, and \localconn it is compact, connected, locally connected, path connected, and locally path connected. Given any four pairwise distinct special points, order them as $x_0,x_1,x_2,x_3$ to satisfy \twodisconn. By (1), for each $i$ there is $\rho$ such that $x_i$ and $x_{i+1}$ are in the same connected component of $X - ({B_\rho}(x_{i-1}) \cup {B_\rho}(x_{i+2}))$ (which is compact), and so there is a path $C_{i}$ from $x_i$ to $x_{i+1}$ not passing through $x_{i-1}$ or $x_{i+2}$.

    Next we must argue that any two of these arcs do not intersect except possibly at their starting and ending points. If not, and, say, $C_i$ and $C_{i+1}$ intersect, let $y$ be the point of intersection. (There are other combinatorial possibilities for the intersection, but a similar argument works in each case.)
    \[\begin{tikzpicture}[
  dot/.style={
    circle,
    fill=black,
    inner sep=2.2pt
  },
  guide/.style={
    densely dotted,
    line width=0.65pt,
    -{Stealth[length=4pt,width=4pt]}
  }
]
  \def\R{2.5}         
  \def\gap{0.18}      
  \pgfmathsetmacro{\Rout}{\R+\gap}

  \coordinate (center) at (0,1.25);
  \coordinate (xi)     at ($(center)+(210:\R)$);
  \coordinate (ximone) at ($(center)+(270:\R)$);
  \coordinate (xipTwo) at ($(center)+(330:\R)$);

  \coordinate (y)      at (0,0);
  \coordinate (xipOne) at (0,1.45);

  \draw (xi) -- (y) -- (xipTwo);

  \draw (xi)
    arc[start angle=210,end angle=330,radius=\R];

  \draw (y)
    .. controls (-0.35,0.42) and (-0.30,1.08) .. (xipOne);
  \draw (y)
    .. controls ( 0.35,0.42) and ( 0.30,1.08) .. (xipOne);

  \draw[guide]
    ($(y)+(-\gap,\gap)$)
      -- ($(xi)+(0,\gap)$)
      arc[start angle=90,end angle=210,radius=\gap]
      arc[start angle=210,end angle=268,radius=\Rout];

  \draw[guide]
    ($(y)+(\gap,\gap)$)
      -- ($(xipTwo)+(-\gap,\gap)$);

  \draw[guide]
    ($(center)+(328:\Rout)$)
      arc[start angle=328,end angle=272,radius=\Rout];

  \node[dot,label=left:{$x_i$}]
    at (xi) {};

  \node[dot,label=below:{$x_{i-1}$}]
    at (ximone) {};

  \node[dot,label=right:{$x_{i+2}$}]
    at (xipTwo) {};

  \node[dot,label=below right:{$y$}]
    at (y) {};

  \node[dot,label=above:{$x_{i+1}$}]
    at (xipOne) {};
\end{tikzpicture}\]
Then $x_{i-1},x_i,x_{i+1},x_{i+2}$ fail to satisfy (2) of \twodisconn because there is a path from $x_i$ to $x_{i+2}$ avoiding $x_{i+1}$ and $x_{i-1}$. A similar argument works for the moreover clause.
\end{proof}

\begin{lemma}
    If $X = \overline{(M,d)}$ satisfies \nondegen, \compact, \conn, \localconn, and \twodisconn, then $X$ is homeomorphic to $S^1$.
\end{lemma}
\begin{proof}
    Since $X = \overline{(M,d)}$ satisfies \compact, \conn, and \localconn it is non-degenerate, compact, connected, locally connected, path connected, and locally path connected.

    Let $x_0,x_1,x_2,x_3$ be four special points, as in Lemma \ref{lem:circle}, and for each $i$, let $C_i$ be the arc from $x_i$ to $x_{i+1}$ with the four arcs forming a homeomorphic copy $C$ of $S^1$. We argue that any other point $a$ must lie on $C$. Suppose not, that $a$ does not lie on $C$.

    Now there is a special point $y$ near $a$, and an arc from $y$ to $a$ disjoint from $C$. There must also be an arc from $y$ to $C$. There are several possibilities, in all of which we get a contradiction:
    \begin{enumerate}
        \item The arc from $y$ to $C$ first meets $C$ at a point other than the $x_i$. Say it meets $C_i$. Then the quadruple $y,x_i,x_{i+1},x_{i+2}$ contradicts Lemma \ref{lem:circle}. (Indeed, there is an arc between any two of $y,x_i,x_{i+1}$ avoiding the third and $x_{i+2}$).

        \item Otherwise, suppose that the arc from $y$ to $C$ first meets $C$ at $x_i$, and that any arc from $y$ to $C$ must pass through $x_i$. Then the quadruple $y,x_i,x_{i-1},x_{i+1}$ contradicts Lemma \ref{lem:circle}.

        \item Otherwise, suppose that in addition to the arc from $y$ to $x_i$, there is also an arc from $y$ that first meets $C$ at $x_{i+2}$. Then the quadruple $y,x_{i-1},x_{i+1},x_{i+2}$ contradicts Lemma \ref{lem:circle} because there is a path between any two of $y,x_{i-1},x_{i+1}$ avoiding the third and $x_{i+2}$.

        \item Finally, suppose that in addition to the arc from $y$ to $x_i$, there is also an arc from $y$ that first meets $C$ at $x_{i+1}$ (or the symmetric case of $x_{i-1}$). Choose a special point $z$ such that there is an arc from $z$ to the middle of $C_i$ that is disjoint from any of the other arcs.  Then the quadruple $y,z,x_i,x_{i+2}$ contradicts Lemma \ref{lem:circle} because there is a path between any two of $y,z,x_{i}$ avoiding the third and $x_{i+2}$.
    \end{enumerate}
    This completes the proof that any other point $a$ must lie on $C$; that is, $C$ is in fact all of $X$ and so $X$ is homeomorphic to $S^1$.
\end{proof}

\begin{theorem}\label{thm:circle-complete}
    The set
    \[ \HCopies(S^1) = \{ Y : Y \cong_{\homeo} S^1 \}\] is $\bfPi^0_4$-Wadge-complete.
\end{theorem}
\begin{proof}
    The other lemmas of this section show that this set is $\Pi^0_4$. To show $\bfPi^0_4$-completeness, we use the same $\Pi^0_4$ set $B$ from Section \ref{sec:pi4}. Let $Y_U$ be the topological space constructed there, with $U \in B$ if and only if $Y_U$ is homeomorphic to $[0,1]$. By Remark \ref{rem:startend}, $Y_U \subseteq \mathbb{R}^{\mathbb{N}}$ always contains $(0,0,0,\ldots)$ and $(1,0,0,\ldots)$ and, in the case that it is homeomorphic to $[0,1]$, these are the endpoints. We can easily construct from $Y_U$ a space $Z_U$ such that $Z_U$ is homeomorphic to $S^1$ if and only if $Y_U$ is homeomorphic to $[0,1]$: In $\mathbb{R} \times \mathbb{R}^\mathbb{N}$, let
    \[ Z_U = \{(t,0,0,\ldots) : t \in [0,1]\} \cup \{(1,t,0,\ldots) : t \in [0,1]\} \cup \{(t,1,0,\ldots) : t \in [0,1]\} \cup \{(0,\mathbf{u}) : \mathbf{u} \in Y_U\}.\]
    Essentially, to construct $Z_U$, we have taken a segment of $S^1$ and replaced it by $Y_U$.
\end{proof}

\section{\texorpdfstring{The real line $\mathbb{R}$}{The real line}}

In this section we will prove that  $\mathbb{R}$ has topological Scott complexity $\bfPi^1_1$.

\begin{theorem}
    $\HCopies(\mathbb{R})$ is $\bfPi^1_1$-Wadge-complete.
\end{theorem}

The proof is in two parts; first we give a $\Pi^1_1$ characterisation of $\mathbb{R}$, and second we prove a $\bfPi^1_1$-hardness result.

\subsection{\texorpdfstring{A $\Pi^1_1$-characterization of $\mathbb{R}$}{A Pi11-characterization of R}}

\begin{theorem}
    $\HCopies(\mathbb{R})$ is $\bfPi^0_4$ within the set of proper metrizations of locally compact spaces, and Borel within the locally compact spaces. In particular, $\HCopies(\mathbb{R})$ is $\Pi^1_1$.
\end{theorem}
\begin{proof}

We begin by working within the locally compact spaces. Given a locally compact but non-compact metric space $(X,d)$, we can form the one-point compactification by adding a new point $\infty$. If the space is proper, i.e., if all bounded closed sets are compact, one can form a metric $\hat{d}$ on the compactification as follows \cite{Mandelkern}. Fix a point $p \in X$, and let $h(x) = \frac{1}{1+d(p,x)}$. Let
\[ \hat{d}(x,y) = \min(d(x,y),h(x)+h(y)), \text{ and } \hat{d}(\infty,x) = h(x).\]
This construction lets us form, from a presentation of a locally compact non-compact space, a presentation of the one-point compactification. Moreover, this is computable.

To apply this construction to an arbitrary locally compact presentation, we first replace $d$ by a proper compatible metric. Using at most a couple of Turing jumps, enumerate finite unions $U_0 \subseteq U_1 \subseteq \cdots$ of basic open balls such that $K_i = \overline{U_i}$ is compact, $K_i \subseteq U_{i+1}$, and the $U_i$ cover $X$. To obtain these sets, use local compactness to enumerate the basic open balls with compact closure, and at each stage a finite subcover of the compact set $K_i$ gives the next $U_{i+1}$. For each $i$, let
\[
    \varphi_i(x) = \frac{d(x,K_i)}{d(x,K_i) + d(x,X \setminus U_{i+1})}.
\]
Then $\varphi_i$ is continuous, is $0$ on $K_i$, and is $1$ outside $U_{i+1}$. The sum $f(x) = \sum_i \varphi_i(x)$ is locally finite, and hence continuous. If $x \notin U_N$, then $\varphi_i(x)=1$ for every $i<N$, so $f(x) \geq N$. Thus $f$ is proper. Now set
\[
    d'(x,y) = \frac{d(x,y)}{1+d(x,y)} + |f(x)-f(y)|.
\]
This is a compatible metric. Moreover, every closed $d'$-ball is contained in a compact sublevel set of $f$, and hence $d'$ is proper. All of the choices and distance computations in this construction are computable using the same bounded number of jumps.

Then a locally compact space is homeomorphic to $\mathbb{R}$ if and only if it is non-compact and its one-point compactification is homeomorphic to $S^1$. (This is because $S^1$ is homogeneous, and $S^1$ minus any point is homeomorphic to $\mathbb{R}$.) The former is a $\Sigma^0_3$ condition and the latter is $\Pi^0_4$ as proved earlier in this paper. The set of locally compact spaces is $\Pi^1_1$ \cite{NiesSolecki}, and thus $\HCopies(\mathbb{R})$ is $\Pi^1_1$.
\end{proof}

Using the fact that $\HCopies([0,1])$ is $\bfPi^0_4$-Wadge-complete, it is easier to prove that $\HCopies(\mathbb{R})$ is $\bfPi^0_4$-Wadge-complete within the set of proper metrizations of locally compact spaces. The proof is essentially the same idea as Theorem \ref{thm:circle-complete} where we take a standard copy of $\mathbb{R}$ and replace a subinterval by the construction of Section \ref{sec:pi4}. The resulting metric is proper.

\begin{proposition}
    $\HCopies(\mathbb{R})$ is $\bfPi^0_4$-Wadge-complete within the set of proper metrizations of locally compact spaces.
\end{proposition}

\subsection{\texorpdfstring{$\bfPi^1_1$-completeness}{P11-completeness}}

Now we must show that it is $\bfPi^1_1$-hard. Given a tree $T \subseteq \omega^{< \omega}$, we will produce a space $X = X_T$ such that $T$ is well-founded if and only if $X$ is homeomorphic to $\mathbb{R}$.

Let $(\sigma_n)$ be a computable listing of $\omega^{< \omega}$ with the property that if $\sigma \prec \tau$ then $\sigma$ appears before $\tau$ in this list.

We work within the ambient space $c_0$ of sequences in $\mathbb{R}^\mathbb{N}$ that converge to $0$, with the norm $\Vert \cdot \Vert_\infty$. Let $\mathbf{u}_\sigma$ (for $\sigma \in \omega^{< \omega}$) and $\mathbf{v}_n$ (for $n \in \omega$) and $\mathbf{w}$ be distinct standard basis elements of the form $(0,0,\ldots,0,1,0,\ldots)$. For convenience, let
\[
\chi_{\sigma} = \begin{cases}
    \frac{1}{|\sigma|+1} & \sigma \in T \\
    1 & \text{otherwise}
\end{cases}
\]
Given $\tau$, let
\[
\mathbf{x}_\tau = \sum_{\sigma \preceq \tau} \chi_\sigma \mathbf{u}_\sigma.
\]
Given $\pi \in \omega^\omega$, note that $\lim_{n \to \infty} \mathbf{x}_{\pi \upharpoonright n}$ converges if and only if $\pi$ is a path through $T$. If it does converge, then it converges to
\[ \mathbf{y}_\pi = \sum_{n} \frac{1}{n+1} \mathbf{u}_{\pi \upharpoonright n}.
\]
Let
\[
\psi_n(t) = \begin{cases}
    0 & t \leq n \\
    2(t-n) & n \leq t \leq n+\frac{1}{2} \\
    2(n+1-t) & n+\frac{1}{2} \leq t \leq n+1 \\
    0 & t \geq n+1
\end{cases}
\]
That is, $\psi_n(t)$ is $0$ at $t = n$, increases linearly to $1$ at $t = n+\frac{1}{2}$, and then decreases linearly back to $0$ at $t = n+1$.

Define a continuous function $p: \mathbb{R} \to c_0$ as follows:
\begin{itemize}
    \item For $t \leq 0$, let $p(t) = e^{-t} \mathbf{w}$.
    \item For $0 \leq t \leq 1$, let
    \[ p(t) = t \mathbf{x}_{\sigma_1} + e^{-t} \mathbf{w} \]
    \item For $1 \leq n \leq t \leq n+1$, let
    \[ p(t) = (n+1-t) \mathbf{x}_{\sigma_n} + (t-n) \mathbf{x}_{\sigma_{n+1}} + \psi_n(t) \mathbf{v}_n + e^{-t} \mathbf{w}.\]
\end{itemize}
Note that $p(n) = \mathbf{x}_{\sigma_n}  + e^{-n} \mathbf{w}$. Essentially, one should think of $p$ as being an interpolation between the sequence of points $\mathbf{x}_{\sigma_n}$, but ``bent'' towards $\mathbf{v}_n$ when passing between $\mathbf{x}_{\sigma_n}$ and $\mathbf{x}_{\sigma_{n+1}}$, and then with an additional decaying term in the direction $\mathbf{w}$. Let $X 
= X_T$ be the closure of the image of $p$. That is, let $M = p[\mathbb{Q}]$ be the countable metric space that is the image of the rationals, and let $X = \overline{M} = \overline{p[\mathbb{R}]}$.

It is easy to see that $p$ is injective, continuous, and computable from $T$. Moreover, $p$ is uniformly continuous on $[0,\infty)$. In fact, $p$ is a homeomorphism onto its image (but its image is not necessarily closed in $c_0$). Consider possible limit points of $p[\mathbb{R}]$ that do not belong to $p[\mathbb{R}]$. Because of the term $e^{-t} \mathbf{w}$, any such point must occur as $t \to \infty$, and because of the ``bends'' in the directions $\mathbf{v}_n$, any such limit point must be the limit of integral points $p(n) = \mathbf{x}_{\sigma_{n}} + e^{-n} \mathbf{w}$ (with the second term going to zero). Thus $p[\mathbb{R}]$ will have a limit point that does not belong to $p[\mathbb{R}]$ if and only if the set $\mathbf{x}_{\sigma_n}$ does, which will happen if and only if $T$ is ill-founded. We verify this more formally.

\begin{lemma}\label{lem:hom-image}
    $p$ is a homeomorphism onto its image.
\end{lemma}
\begin{proof}
    We must show given a sequence $t_i \in \mathbb{R}$, and $s \in \mathbb{R}$, that $t_i \to s$ if and only if $p(t_i) \to p(s)$. One direction follows from the fact that $p$ is continuous, and the other---that if $p(t_i) \to p(s)$ then $t_i \to s$---can be seen using the term $e^{-t} \mathbf{w}$ that always appears in $p(t)$. 
\end{proof}

We now argue that $p$ is a homeomorphism onto the \textit{closure} of its image if and only if its image is closed if and only if $T$ is well-founded.

\begin{lemma}
    Suppose that $T$ is well-founded. Then the image of $p$ is closed.
\end{lemma}
\begin{proof}
    Suppose that $p(t_i) \to \mathbf{z}$, and $t_i$ does not converge. Write
    \[ \mathbf{z} = \sum_\sigma a_\sigma \mathbf{u}_\sigma + \sum_n b_n \mathbf{v}_n + c \mathbf{w}.\]
    We must have $c = \lim_{i \to \infty} e^{- t_i}$ and so, since the $t_i$ do not converge, $c = 0$ and $t_i \to \infty$. Since $t_i \to \infty$, we must have that $b_n = 0$ for all $n$.

    Thus
    \[ \mathbf{z} = \sum_\sigma a_\sigma \mathbf{u}_\sigma.\]
    By passing to a subsequence, we may assume that for each $i$ there is $n_i \in \mathbb{N}$ such that $t_i \in [n_i-\frac{1}{2},n_i+\frac{1}{2}]$, and moreover we may assume that the $n_i$ form a strictly increasing sequence. Also, since $\mathbf{z}$ does not involve any of the $\mathbf{v}_n$, it must be that
    \[ \lim_{i \to \infty} |n_i - t_i| = 0.\]
    So in fact we may assume that $t_i = n_i$. Since $p(n_i) = \mathbf{x}_{\sigma_{n_i}} + e^{-n_i}\mathbf{w}$ and $e^{-n_i} \to 0$, the sequences $(p(n_i))$ and $(\mathbf{x}_{\sigma_{n_i}})$ have the same limit.

    Now it is not hard to see that the only way to have $\lim_{i \to \infty} p(n_i)$ exist is to have the $\sigma_{n_i}$ have longer and longer initial segments that agree with a path $\pi$ through $T$, and $\lim_{i \to \infty} p(n_i) = \mathbf{y}_\pi$. Thus $T$ has a path.

    So we conclude that if $T$ is well-founded, then $\lim_{i \to \infty} p(t_i)$ converges if and only if $t_i$ converges; in particular, $p$ is a homeomorphism of $\mathbb{R}$ onto its image and its image is closed in $c_0$.
\end{proof}

Thus if $T$ is well-founded, $X = \overline{M} = \overline{p[\mathbb{R}]} = p[\mathbb{R}]$, and this is homeomorphic to $\mathbb{R}$ by Lemma \ref{lem:hom-image}.

\begin{lemma}
    Suppose that $T$ has an infinite path $\pi$. Then there is an increasing integer sequence $n_i$ such that $\lim_{i \to \infty} p(n_i)$ converges.
\end{lemma}
\begin{proof}
    Let $\pi$ be a path through $T$, and let $n_i$ be an increasing integer sequence such that $\sigma_{n_i} = \pi \res i$. Then $p(n_i) = \mathbf{x}_{\pi \upharpoonright i} + e^{-n_i} \mathbf{w}$, and 
    \[ \lim_{i \to \infty} p(n_i) = \mathbf{y}_\pi = \sum_{n} \frac{1}{n+1} \mathbf{u}_{\pi \upharpoonright n}.\]
    Since the $\mathbf{w}$ coordinate of $\mathbf{y}_\pi$ is $0$, while the $\mathbf{w}$ coordinate of every $p(t)$ is $e^{-t}>0$, we have $\mathbf{y}_\pi \notin p[\mathbb{R}]$. Thus $p[\mathbb{R}]$ is a proper dense subset of $X$.
    Thus $p$ maps a non-convergent sequence to a convergent sequence, and so cannot be a homeomorphism onto the closure of its image.
\end{proof}

If $T$ is not well-founded, then $X = \overline{M} = \overline{p[\mathbb{R}]}$ contains a homeomorphic image of $\mathbb{R}$ as a proper dense subset. But $\mathbb{R}$ does not contain a homeomorphic image of $\mathbb{R}$ as a proper dense subspace (indeed, $\mathbb{R}$ has no proper connected dense subsets). Thus $X$ is not homeomorphic to $\mathbb{R}$.

\bibliography{References}
\bibliographystyle{alpha}
	
\end{document}